\documentclass[12pt]{amsart}
\usepackage{amssymb}
\usepackage{txfonts}
\usepackage{eucal}

\usepackage[all]{xy}

\usepackage{tikz}

\usepackage{graphicx}
\usepackage{float}

\usepackage{xspace}

\usepackage[a4paper,body={16.3cm,22.8cm},centering]{geometry}
\usepackage[colorlinks,final,backref=page,hyperindex,pdftex]{hyperref}

\newcommand{\nc}{\newcommand}
\newcommand{\delete}[1]{}

\definecolor{vert}{rgb}{0.,0.5,0.}

\newcommand{\opdend}{\mathbf{Dend}_\Omega^2}
\newcommand{\opdup}{\mathbf{Dup}_\Omega^2}

\newcommand{\Z}{\mathbb{Z}}
\newcommand{\dimK}{\mathrm{dim}_{\mathbf k}}
\newcommand{\cat}{\mathrm{Cat}}

\nc{\mlabel}[1]{\label{#1}}  
\nc{\mcite}[1]{\cite{#1}}  
\nc{\mref}[1]{\ref{#1}}  
\nc{\mbibitem}[1]{\bibitem{#1}} 

\delete{
\nc{\mlabel}[1]{\label{#1}  
{\hfill \hspace{1cm}{\small\tt{{\ }\hfill(#1)}}}}
\nc{\mcite}[1]{\cite{#1}{\small{\tt{{\ }(#1)}}}}  
\nc{\mref}[1]{\ref{#1}{{\tt{{\ }(#1)}}}}  
\nc{\mbibitem}[1]{\bibitem[\bf #1]{#1}} 
}

\newtheorem{theorem}{Theorem}[section]
\newtheorem{prop}[theorem]{Proposition}

\newtheorem{coro}[theorem]{Corollary}

\theoremstyle{definition}
\newtheorem{defn}[theorem]{Definition}
\newtheorem{remark}[theorem]{Remark}
\newtheorem{exam}[theorem]{Example}
\newtheorem{prop-def}{Proposition-Definition}[section]

\newcommand\cal[1]{\mathcal{\MakeUppercase{#1}}}

\newcommand\alphlist{a,b,c,d,e,f,g,h,i,j,k,l,m,n,o,p,q,r,s,t,u,v,w,x,y,z}
\newcommand\Alphlist{A,B,C,D,E,F,G,H,I,J,K,L,M,N,O,P,Q,R,S,T,U,V,W,X,Y,Z}
\newcommand\getcmds[3]{\expandafter\newcommand\csname #2#1\endcsname{#3{#1}}}
\makeatletter
\@for\x:=\alphlist\do{\expandafter\getcmds\expandafter{\x}{cal}{\cal}}      
\@for\x:=\alphlist\do{\expandafter\getcmds\expandafter{\x}{frak}{\mathfrak}}
\@for\x:=\Alphlist\do{\expandafter\getcmds\expandafter{\x}{frak}{\mathfrak}}
\makeatother

\nc{\bfk}{{\bf k}}

\font\cyr=wncyr10

\newfont{\scyr}{wncyr10 scaled 550}
\nc{\sha}{\mbox{\cyr X}}
\nc{\ssha}{\mbox{\bf \scyr X}}

\nc{\Id}{\mathrm{Id}}
\nc{\lbar}[1]{\overline{#1}}
\nc{\ot}{\otimes}
\nc{\dep}{\mathrm{dep}}

\nc{\tred}[1]{\textcolor{red}{#1}} \nc{\tgreen}[1]{\textcolor{green}{#1}}
\nc{\tblue}[1]{\textcolor{blue}{#1}} \nc{\tpurple}[1]{\textcolor{purple}{#1}}

\nc{\li}[1]{\tpurple{\underline{Li: }#1 }}
\nc{\liadd}[1]{\tpurple{#1}}
\nc{\xing}[1]{\tblue{\underline{Xing: }#1 }}
\nc{\yuan}[1]{\tred{\underline{Yuan: }#1 }}
\nc{\markus}[1]{\tred{\underline{Markus: } #1}}
\nc{\dominique}[1]{\tpurple{\underline{Dominique: }#1 }}

\usepackage{tikz}
\usetikzlibrary{calc,shapes.geometric}
\tikzset{
baseon/.style={baseline={($(#1)+(0,-0.58ex)$)}},
baseon/.default=current bounding box.center,
every picture/.style=baseon,
lst/.style={},
dst/.style={fill,circle,inner sep=1.5pt,outer sep=0pt},
ddst/.style={diamond,draw,inner sep=1pt},
eest/.style={ellipse,draw,inner sep=1pt,minimum size=2ex},
}
\newcommand\treeo[2][]{\tikz[x=0.7cm,y=0.7cm,line width=0.15ex,
every node/.style={font=\scriptsize,inner sep=2pt,label distance=1pt},#1]{%
\coordinate (o) at (0,0);#2}}%
\def\zzz#1`#2...#3`#4...#5`#6@{%
--++(#1)
node[dst,label={#5:$#6$},name=#2]{}
node[midway,auto,#3]{$#4$}
}
\def\ddd#1`#2`#3@{+(#1)node[ddst,name=#2]{$#3$}}
\def\eee#1`#2`#3@{+(#1)node[eest,name=#2]{$#3$}}
\def\xxx#1`#2@{node[midway,auto,inner sep=1pt,#1]{$#2$}}
\def\pp#1`#2`#3@{node[dst,label={#2:$#3$},pos=#1]{}}
\def\oo#1`#2`#3@{\path (o) node[dst,label={#2:$#3$},name=o,#1]{};}
\def\eoo#1`#2@{\node[eest,name=o,#1] at (o) {$#2$};}
\newlength\xch
\newsavebox\dbox
\sbox\dbox{\tikz{\fill (0,0) circle (0.05cm);}}
\newif\ifqdd
\newif\ifzdd
\newcommand\lstyle{}
\newcommand\cddf[3]{%
\coordinate (#2) at ($(#1)+(#3)$);
\draw[\lstyle] (#1)--(#2);
\ifqdd\node at (#1) {\usebox\dbox};\fi
\ifzdd\node at (#2) {\usebox\dbox};\fi}
\newcommand\cdx[4][1]{\cddf{#2}{#3}{#4:#1*\xch}}
\newcommand\cdl[2][1]{\cdx[#1]{#2}{#2l}{135}}
\newcommand\cdr[2][1]{\cdx[#1]{#2}{#2r}{45}}
\newcommand\cdlr[2][1]{%
\foreach \i in {#2} {\cdl[#1]{\i}\cdr[#1]{\i}}}
\newcommand\cdb[2][1]{\cdx[#1]{#2}{#2b}{-90}}
\newcommand\ocdx[6][1]{%
\node[draw,circle,minimum size=2pt,label={#6:$#5$}]
(#3) at ($(#2)+(#4:#1*\xch)$) {};
\draw (#2)--(#3);}

\newcommand\scopeclip[1]{\begin{scope}
\clip(-1.1,-0.5)rectangle(1.1,1);#1\end{scope}}
\newcommand\XX[2][]{%
\tikz[line width=0.15ex,x=0.5cm,y=0.5cm,baseline,inner sep=1.5pt,
every node/.style={font=\scriptsize},#1]{
\scopeclip{\draw (135:1.5)--(0,0)--(45:1.5) (0,-0.5)--(0,0);}#2}}
\newcommand\xx[3]{%
\scopeclip{\draw(#1/10,#2/10)--+(#3*45:2.5);}}
\newcommand\xxl[2]{\xx{#1}{#2}3}
\newcommand\xxr[2]{\xx{#1}{#2}1}

\newcommand\xxh[6]{
\draw(#1/10,#2/10)+(0.5*#3*45+0.5*#4*45:#6) node[above] {$#5$};}
\newcommand\xxhu[4][0.15]{\xxh{#2}{#3}13{#4}{#1}}
\newcommand\stree[1]{\XX{\xxhu[0.25]00{#1}}}

\nc{\dnx}{\Delta_n A} \nc{\dx}{\Delta A} \nc{\dgp}{{\rm deg_{P}}}
\nc{\dgt}{{\rm deg_{T}}} \nc{\dg}{{\rm deg}} \nc{\ida}{ID($A$)} \nc{\tu}{\tilde{u}} \nc{\tv}{\tilde{v}}
\nc{\nr}{\calr_n} \nc{\nz}{\calz_n} \nc{\fun}{\cala_{n,d}}
 \nc{\fbase}{\calb} \nc{\LF}{\mathrm{RF}} \nc{\FFA}{\mathrm{LF}} \nc{\irr}{\mathrm{Irr}}
 \nc{\result}{\bfk\mathrm{Irr}(S_n)}  \nc{\I}{I_{\mathrm{ID},n}^0}
 \nc{\nrs}{\calr_n^ \circledast } \nc{\ii}{\mathrm{I}} \nc{\iii}{\mathrm{II}}
\nc{\intl}{{\rm int}}\nc{\ws}[1]{{#1}}\nc{\deleted}[1]{\delete{#1}}\nc{\plas}{placements\xspace}

\nc{\bim}[1]{#1}  \nc{\shaop}{\sha_{\Omega}^{+}}  \nc{\shao}{\sha_{\Omega}}
\nc{\bbim}[2]{#1 #2} \nc{\bbbim}[2]{#1,\, #2} \nc{\RBF}{{\rm RBF}}
\nc{\frb}{F_{\RB}} \nc{\shaf}{\ssha_{\tiny{\Omega}}} \nc{\sham}{\diamond_{\tiny{\Omega}}}
\nc{\lf}{\lfloor} \nc{\rf}{\rfloor} \nc{\shan}{\ssha_{\lambda}}
\nc{\rlex}{{\rm {lex}}} \nc{\bb}{\Box} \nc{\ra}{\rightarrow}
\nc{\e}{{\rm {e}}}
\nc{\DD}{\mathrm{D}(X,\,\Omega)}
\nc{\bre}{\mathrm{bre}}
\nc{\dec}{\mathrm{dec}}
\nc{\type}{\mathrm{type}}

\long\def\ignore#1{}
\nc{\cP}{\hskip 2pt ${\mathcal P}$\put(-3.5,4){\circle{11}}\hskip 5pt}
\nc{\cR}{\hskip 2pt ${\mathcal R}$\put(-4.3,3.7){\circle{11}}\hskip 5pt}
\nc{\cpi}{\hskip 2pt ${\pi}$\put(-3.5,3){\circle{9}}\hskip 5pt}
\def\fleche#1{\mathop{\hbox to #1 mm{\rightarrowfill}}\limits}
\def\gfleche#1{\mathop{\hbox to #1 mm{\leftarrowfill}}\limits}
\def\inj#1{\mathop{\hbox to #1 mm{$\lhook\joinrel$\rightarrowfill}}\limits}
\def\ginj#1{\mathop{\hbox to #1 mm{\leftarrowfill$\joinrel\rhook$}}\limits}
\def\surj#1{\mathop{\hbox to #1 mm{\rightarrowfill\hskip 2pt\llap{$\rightarrow$}}}\limits}
\def\gsurj#1{\mathop{\hbox to #1 mm{\rlap{$\leftarrow$}\hskip 2pt \leftarrowfill}}\limits}
\def\wwt#1{\widetilde{\widetilde{#1}}}
\def\diagramme #1{\vskip 4mm \centerline {#1} \vskip 4mm}
\def \restr#1{\mathstrut_{\textstyle |}\raise-6pt\hbox{$\scriptstyle #1$}}
\def \srestr#1{\mathstrut_{\scriptstyle |}\hbox to -1pt{}\raise-4pt\hbox{\hskip 1pt$\scriptscriptstyle #1$}}

\def\Ve#1,#2,#3;{\vee_{#1,\,(#2,\,#3)}}
\def\bigv#1;#2;#3;{\bigvee\nolimits_{#1}^{#2;\,#3}}
\def\ba#1;{\backslash_{#1}\,}
\def\bd#1;{\mathbin{\,_{#1}/}}
\def\ct#1,#2;{\cdot_{#1,\,#2}}

\nc{\la}{\leftarrow}
\nc{\DS}{\mathbf{DS}}
\nc{\br}{\blacktriangleright}
\nc{\bl}{\blacktriangleleft}
\newcommand{\DUS}{\mathbf{DuS}}
\newcommand{\EDUS}{\mathbf{EDuS}}
\nc{\tr}{\triangleright}
\nc{\tl}{\triangleleft}
\nc{\DDF}{\mathrm{DD}(X,\,\Omega)}\nc{\DTF}{\mathrm{DT}(X,\,\Omega)}
\begin{document}

\title[Families of algebraic structures]{Families of algebraic structures}
%
%
\author{Lo\"\i c Foissy}
\address{ULCO}
\email{loic.foissy@univ-littoral.fr}
\author{Dominique Manchon}
\address{Laboratoire de Math\'ematiques Blaise Pascal,
CNRS--Universit\'e Clermont-Auvergne,
3 place Vasar\'ely, CS 60026,
F63178 Aubi\`ere, France}
\email{Dominique.Manchon@uca.fr}
\author{Yuanyuan Zhang} \address{School of Mathematics and Statistics,
Lanzhou University, Lanzhou, 730000, P. R. China}
\email{zhangyy17@lzu.edu.cn}

\date{\today}

\begin{abstract}
We give a general account of family algebras over a finitely presented linear operad, this operad together with its presentation naturally defining an algebraic structure on the set of parameters.
\end{abstract}

\subjclass[2010]{
16W99, 
16S10, 
08B20, 
16T30,  
17D25      
}

\keywords{Monoidal categories, graded objects, operads, colored operads, dendriform algebras, duplicial algebras, pre-Lie algebras, typed decorated planar binary trees, diassociative semigroups.}

\maketitle

\tableofcontents

\setcounter{section}{0}

\allowdisplaybreaks
\section{Introduction}\mlabel{sec:rota}
The first family algebra structures appeared in the literature in 2007: a natural example of Rota-Baxter family algebras of weight $-1$ was given by J.~Gracia-Bond\'\i a, K.~Ebrahimi-Fard and F.~Patras in a paper on Lie-theoretic aspects of renormalization~\cite[Proposition~9.1]{EGP} (see also~\cite{DK}). The notion of Rota-Baxter family itself was suggested to the authors by Li Guo (see Footnote after Proposition 9.2 therein), who started a systematic study of these Rota-Baxter family algebras in \cite{Guo09}, including the more general case of weight $\lambda$. They are associative algebras $P$ over some field $\bfk$ together with a collection $(P_\omega)_{\omega\in\Omega}$ of linear endomorphisms indexed by a semigroup $\Omega$ such that the Rota-Baxter family relation
\begin{equation*}
P_{\alpha}(a)P_{\beta}(b)=P_{\alpha\beta}\bigl( P_{\alpha}(a)b  + a P_{\beta}(b) + \lambda ab \bigr)
\end{equation*}
 holds for any $a, b \in R$ and $\alpha,\, \beta \in \Omega$. The example in \cite{EGP} is given by the momentum renormalization scheme: here $\Omega$ is the additive semigroup of non-negative integers, and the operator $P_\omega$ associates to a Feynman diagram integral its Taylor expansion of order $\omega$ at vanishing exterior momenta.  The simplest example we can provide, derived from the minimal subtraction scheme, is the algebra of Laurent series $R=\bfk[z^{-1},z]]$, where, for any $\omega\in\Omega=\mathbb Z$, the operator $P_\omega$ is the projection onto the subspace $R_{<\omega}$ generated by $\{z^k,\,k<\omega\}$ parallel to the supplementary subspace $R_{\ge \omega}$ generated by $\{z^k,\,k\ge\omega\}$.\\
 
 Other families of algebraic structures appeared more recently: dendriform and tridendriform family algebras \mcite{ZG, ZGM, Foi20}, pre-Lie family algebras \cite{MZ},... The principle consists in replacing each product of the structure by a family of products, so that the operadic relations (Rota-Baxter, dendriform, pre-Lie,...) still hold in a ``family" version taking the semigroup structure of the parameter set into account. An important step in understanding family structures in general has been recently done by M. Aguiar, who defined family $\mathcal P$-algebras for any linear operad $\mathcal P$ \cite{A2020}. The semigroup $\Omega$ of parameters must be commutative unless the operad is non-sigma. An important point is that any $n$-ary operation gives rise to a family of operations parametrized by $\Omega^n$. In particular, the natural way to ``familize" a binary operation requires two parameters.\\
 
 The first author recently described a variant of one-parameter dendriform family algebras for which the set $\Omega$ of parameters in endowed with the very rich structure of extended diassociative semigroup \cite{Foi20}. We follow here the same path for two-parameter dendriform family algebras, where $\Omega$ is a now a (non-extended) diassociative semigroup. This suggests that the natural algebraic structure of $\Omega$ is determined in some way by the operad one starts with. This appears to be the case: we define family $\mathcal P$-algebras for any finitely presented linear operad $\mathcal P$, in a way which depends on the presentation chosen. The definition makes sense when the parameter set $\Omega$ is endowed with a \cP-algebra structure, where \cP is a set operad determined by $\mathcal P$ and its presentation. Following the lines of M.~Aguiar, we define family $\mathcal P$-algebras indexed by $\Omega$ as \textsl{uniform $\Omega$-graded $\mathcal P$-algebras}. The notion of $\Omega$-graded $\mathcal P$-algebra, when $\Omega$ is a \cP-algebra, is defined via {color-mixing operads}, which are generalizations of the current-preserving operads of \cite{S2014}.\\
 
 The paper is organized as follows: we investigate one-parameter and two-parameter dendriform family algebras over a fixed base field $\bfk$ in some detail in Section \ref{sect:dd}, as well as their duplicial counterparts. We give the definition of a two-parameter dendriform family algebra $A$ indexed by a diassociative semigroup $\Omega$. This structure on the index set naturally appears when one asks $A\otimes \bfk\Omega$ to be a graded dendriform algebra, in the sense that the homogeneous components $(A\otimes\bfk\omega)_{\omega\in\Omega}$ are respected. The further structure of extended diassociative semigroup (EDS) appears for one-parameter dendriform family algebras \cite{Foi20}. The situation for duplicial family is similar but simpler, due to the fact that the duplicial operad is a set operad. The structure which appears on $\Omega$ is that of duplicial semigroup. Again, a further structure of extended duplicial semigroup (EDuS) appears in the one-parameter version.\\
 
 We also give an example of two-parameter duplicial family algebra in terms of planar binary trees with $\Omega\times\Omega$-typed edges, and we prove that planar binary trees with $\Omega$-typed edges provide free one-parameter $\Omega$-duplicial algebras for any EDuS $\Omega$. To conclude this section, we give the generating series of the dimensions of the free two-parameter duplicial (or dendriform) family algebra with one generator, when the parameter set $\Omega$ is finite.\\
 
 We give a reminder of colored operads in Joyal's species formalism \cite{J1234} in Section \ref{sect:operadic}, and give a brief account of graded objects. Following a crucial idea in \cite{A2020}, we describe the uniformization functor $\mathcal U$ from ordinary (monochromatic) operads to colored operads (resp., with the same notations, from a suitable monoidal category to its graded version), and its left-adjoint, the completed forgetful functor $\overline{\mathcal F}$.\\
 
 In Section \ref{sect:prelie}, we study the pre-Lie case in some detail. The pre-Lie operad $\mathcal P$ gives rise to four different set operads, namely the associative operad, the twist-associative operad governing Thedy's rings with $x(yz)=(yx)z$ \cite{T1967}, an operad built from corollas governing rings with $x(yz)=y(xz)$ and $(xy)z=(yx)z$, and finally the Perm operad governing rings with $x(yz)=y(xz)=(xy)z=(yx)z$, i.e. set-theoretical Perm algebras. This last operad is a quotient of the three others and gives rise to family pre-Lie algebras. Finally, color-mixing operads and the general definition of $\Omega$-family algebras are given in Section \ref{sect:multigraded}.\\
 
\noindent {\bf Notation:} In this paper, we fix a field $\bfk$ and assume that an algebra is a \bfk-algebra. The letter $\Omega$ will denote a set of indices, which will be endowed with various structures throughout the article.
\section{Dendriform and duplicial family algebras}\label{sect:dd}
\subsection{Two-parameter $\Omega$-dendriform algebras}\label{par:omegadend}
First, we borrow some concepts from the first author's recent article~\cite{Foi20}.
\begin{defn}\mlabel{defn:dia}
{\bf A diassociative semigroup} is a triple $(\Omega,\la,\ra)$, where $\Omega$ is a set and $\la,\ra:\Omega\times\Omega\ra\Omega$ are maps such that, for any $\alpha,\beta,\gamma\in\Omega:$
\begin{align}
(\alpha\la\beta)\la\gamma&=\alpha\la(\beta\la\gamma)=\alpha\la(\beta\ra\gamma),\\
(\alpha\ra\beta)\la\gamma&=\alpha\ra(\beta\la\gamma),\\
(\alpha\ra\beta)\ra\gamma&=(\alpha\la\beta)\ra\gamma=\alpha\ra(\beta\ra \gamma).
\end{align}
\end{defn}

\begin{exam}~\cite{Foi20}
\begin{enumerate}
\item If $(\Omega, \circledast )$ is an associative semigroup, then $(\Omega, \circledast , \circledast )$ is a diassociative semigroup.
\item Let $\Omega$ be a set. For $\alpha,\beta \in \Omega,$ let
\begin{align*}
&\alpha \la \beta=\alpha,&\alpha \rightarrow \beta=\beta.
\end{align*}
Then $(\Omega,\la,\ra)$ is a diassociative semigroup denoted by $\DS(\Omega)$.
\end{enumerate}
\end{exam}

\begin{defn}\mlabel{defn:pare}
Let $\Omega$ be a diassociative semigroup with two products $\la$ and $\ra$. {\bf A two-parameter $\Omega$-dendriform algebra} is a family $(A,(\prec_{\alpha,\beta},\succ_{\alpha,\beta})_{\alpha,\beta\in\Omega})$ where $A$ is a vector space and
$$\prec_{\alpha,\beta},\succ_{\alpha,\beta}:A\ot A\ra A$$
are bilinear binary products such that for any $x,y,z\in A$, for any $\alpha,\beta\in\Omega,$
\begin{align}
(x\prec_{\alpha,\beta}y)\prec_{\alpha\la\beta,\gamma}z
&=x\prec_{\alpha,\beta\la\gamma}(y\prec_{\beta,\gamma} z)+x\prec_{\alpha,\beta\ra\gamma}(y\succ_{\beta,\gamma}z),\mlabel{eq:prec}\\
(x\succ_{\alpha,\beta}y)\prec_{\alpha\ra\beta,\gamma}z&=x\succ_{\alpha,\beta\la\gamma}
(y\prec_{\beta,\gamma}z),\mlabel{eq:prsu}\\
x\succ_{\alpha,\beta\ra\gamma}(y\succ_{\beta,\gamma}z)
&=(x\succ_{\alpha,\beta}y\succ_{\alpha\ra\beta,\gamma}z
+(x\prec_{\alpha,\beta}y)\succ_{\alpha\la\beta,\gamma}z.\mlabel{eq:succ}
\end{align}
\end{defn}

\begin{remark}
\begin{enumerate}
\item If $\Omega$ is a set, we recover the definition of a two-parameter version of matching dendriform algebras \cite{GGZ} when we consider $(\Omega,\la,\ra)$ as a diassociative semigroup, that is, for any $\alpha,\beta\in\Omega,$
    $$\alpha\la\beta=\alpha,\quad \alpha\ra\beta=\beta.$$
\item If $(\Omega, \circledast )$ is a semigroup, we recover the definition of two-parameter dendriform family algebras given in ~\cite{A2020} when we consider
    $$\alpha\la\beta=\alpha\ra\beta=\alpha \circledast \beta.$$
\end{enumerate}
\end{remark}
Two-parameter $\Omega$-dendriform algebras are related to dendriform algebras
and diassociative semigroups by the following proposition:

\begin{prop} \label{propequivalence}
Let $\Omega$ be a set with two binary operations $\la$ and $\ra$.
\begin{enumerate}
\item Let $A$ be a $\bfk$-vector space and let
$$\prec_{\alpha,\beta},\succ_{\alpha,\beta}:A\ot A\ra A$$
be two families of bilinear binary products indexed by $\Omega\times\Omega$. We define products $\prec$ and $\succ$ on the space $A\otimes \bfk\omega$ by:
\begin{eqnarray}
(x\otimes\alpha)\prec(y\otimes\beta)&=&(x\prec_{\alpha,\beta}y)\otimes (\alpha\leftarrow\beta),\label{grad-one1}\\
(x\otimes\alpha)\succ(y\otimes\beta)&=&(x\succ_{\alpha,\beta}y)\otimes (\alpha\rightarrow\beta).\label{grad-two1}
\end{eqnarray}
If $(A\otimes \bfk\omega,\prec,\succ)$ is a dendriform algebra, then \eqref{eq:prec}, \eqref{eq:prsu} and \eqref{eq:succ} hold.\label{propequivalence-a}
\item The following conditions are equivalent:
\begin{enumerate}
\item For any $\big(A,(\prec_{\alpha,\beta},\succ_{\alpha,\beta})
_{\alpha,\beta\in\Omega}\big)$ where $A$ is a $\bfk$-vector space and where \eqref{eq:prec}, \eqref{eq:prsu} and \eqref{eq:succ} hold, the vector space $A\otimes \bfk\Omega$ endowed with the binary operations $\prec$ and $\succ$ defined by
\eqref{grad-one1} and \eqref{grad-two1} is a dendriform algebra,
\item $(\Omega,\la,\ra)$ is a diassociative semigroup,
\item Any $\big(A,(\prec_{\alpha,\beta},\succ_{\alpha,\beta})
_{\alpha,\beta\in\Omega}\big)$ where $A$ is a $\bfk$-vector space and where \eqref{eq:prec}, \eqref{eq:prsu} and \eqref{eq:succ} hold is a two-parameter $\Omega$-dendriform algebra.
\end{enumerate}
\label{propequivalence-b}
\end{enumerate}
\end{prop}

\begin{proof}
\ref{propequivalence-a}. Let us consider the three dendriform axioms:
\begin{align*}
\Big((x\ot\alpha)\prec(y\ot\beta)\Big)\prec(z\ot\gamma)&=(x\ot\alpha)
\prec\Big((y\ot\beta)\prec(z\ot\gamma)+(y\ot\beta)\succ(z\ot\gamma)\Big)\\
\Big((x\ot\alpha)\succ(y\ot\beta)\Big)\prec(z\ot\gamma)&=(x\ot\alpha)
\succ\Big((y\ot\beta)\prec(z\ot\gamma)\Big)\\
(x\ot\alpha)\succ\Big((y\ot\beta)\succ(z\ot\gamma)\Big)&=
\Big((x\ot\alpha)\succ(y\ot\beta)+(x\ot\alpha)\prec(y\ot\beta)\Big)\succ(z\ot\gamma).
\end{align*}
The first one gives:
\begin{align}
\label{eqdendriforme1}
(x\prec_{\alpha,\beta} y)\prec_{\alpha \la \beta,\gamma}z \otimes (\alpha \la \beta) \la \gamma
&=x\prec_{\alpha,\beta \la \gamma}(y\prec_{\beta,\gamma} z)\otimes \alpha \la (\beta\la \gamma)\\
\nonumber&+x\prec_{\alpha,\beta \ra \gamma}(y\succ_{\beta,\gamma} z)\otimes \alpha \ra (\beta\la \gamma).
\end{align}
Let $f:\bfk\Omega \longrightarrow \bfk$ be the linear map sending any $\delta \in \Omega$ to $1$.
Applying $Id_A\otimes f$ to both sides of (\ref{eqdendriforme1}), we obtain (\ref{eq:prec}).
Similarly, the second dendriform axiom gives (\ref{eq:prsu}) and the last one gives (\ref{eq:succ}).\\

\ref{propequivalence-b}. $(i)\Longrightarrow (ii)$. Let us consider the free 2-parameter $\Omega$-dendriform algebra $A$ on three generators
$x$, $y$ and $z$ (from the operad theory, such an object exists). Let us fix $\alpha$, $\beta$ and $\gamma$ in $\Omega$.
According to the relations defining  2-parameter $\Omega$-dendriform algebras,
$x\prec_{\alpha,\beta \la \gamma}(y\prec_{\beta,\gamma} z)$ and $x\prec_{\alpha,\beta \ra \gamma}(y\succ_{\beta,\gamma} z)$
are linearly independent in $A$. Let $g:A\longrightarrow \bfk$ be a linear map such that
\begin{align*}
g\big(x\prec_{\alpha,\beta \la \gamma}(y\prec_{\beta,\gamma} z)\big)&=1,\\
g\big(x\prec_{\alpha,\beta \ra \gamma}(y\succ_{\beta,\gamma} z)\big)&=0.
\end{align*}
Applying $g\otimes Id_{\bfk\Omega}$ on both sides of (\ref{eqdendriforme1}), we obtain that there exists a scalar
$\lambda$ such that $\lambda(\alpha \la \beta) \la \gamma=\alpha \la (\beta\la \gamma)$.
As $(\alpha \la \beta) \la \gamma$ and $\alpha \la (\beta\la \gamma)$ are both elements of $\Omega$, necessarily
$\lambda=1$. Using a linear map $h:A\longrightarrow\bfk$ such that
\begin{align*}
h\big(x\prec_{\alpha,\beta \la \gamma}(y\prec_{\beta,\gamma} z)\big)&=0,\\
h\big(x\prec_{\alpha,\beta \ra \gamma}(y\succ_{\beta,\gamma} z)\big)&=1,
\end{align*}
we obtain that $(\alpha \la \beta) \la \gamma=\alpha \ra (\beta\la \gamma)$.
The other axioms of diassociative semigroups are obtained in the same way from the second and third
dendriform axioms.

$(ii)\Longrightarrow (i)$. If $(ii)$ holds, (\ref{eqdendriforme1}) immediately implies that the first dendriform axiom
is satisfied for any $A$. The second and third dendriform axioms are proved in the same way. Finally, $(ii)\Longleftrightarrow (iii)$ is obvious.
\end{proof}

\begin{remark}
We recover the ordinary (i.e. one-parameter) definitions of matching dendriform algebras in~\cite{GGZ} (resp. dendriform family algebras in~\cite{ZGM}) from Definition~\mref{defn:pare} if $\Omega$ is a set with diassociative semigroup structure given by $\alpha\leftarrow\beta=\alpha$ and $\alpha\rightarrow\beta=\beta$ for any $\alpha,\beta\in\Omega$ (resp. if $(\Omega,\circledast)$ is a semigroup with diassociative semigroup structure given by $\alpha\leftarrow\beta=\alpha\rightarrow\beta=\alpha\circledast\beta$ for any $\alpha,\beta\in\Omega$), if we suppose that $\prec_{\alpha,\beta}$ depends only on $\beta$ and $\succ_{\alpha,\beta}$ depends only on $\alpha$:
\begin{equation*}
\prec_{\alpha,\beta}=\prec_\beta,\quad\succ_{\alpha,\beta}=\succ_\alpha\,\text{ for }\, \alpha,\beta\in\Omega.
\end{equation*}
A general definition of one-parameter dendriform family algebras encompassing both \cite{GGZ} and \cite{ZGM} has been recently proposed by the first author. This requires an extra structure of \textsl{extended diassociative semigroup} (in short, EDS) on the index set $\Omega$, namely two extra binary products $\lhd,\rhd$ subject to a bunch of compatibility axioms between themselves and with the diassociative structure $(\leftarrow,\rightarrow)$ \cite{Foi20}. More precisely, if $(\Omega,\la,\ra,\lhd,\rhd)$ is an extended diassociative semigroup and  $\big(A,(\prec_{\alpha},\succ_{\alpha})_{\alpha\in\Omega}\big)$ is a one-parameter $\Omega$-dendriform algebra in the sense of \cite{Foi20}, then it is a 2-parameter $\Omega$-dendriform algebra with the products
\begin{align*}
\prec_{\alpha,\beta}&=\prec_{\alpha \lhd \beta},&
\succ_{\alpha,\beta}&=\succ_{\alpha \rhd \beta}.
\end{align*}
This is an immediate consequence of Proposition 18 of \cite{Foi20} and Proposition \ref{propequivalence}-\ref{propequivalence-a}.
\end{remark}

\subsection{Two-parameter $\Omega$-duplicial algebras}\label{par:omegadup}
\noindent We can mimick step by step the construction of Paragraph \ref{par:omegadend}:
\begin{defn}\mlabel{defn:dus}
{\bf A duplicial semigroup} is a triple $(\Omega,\la,\ra)$, where $\Omega$ is a set and $\la,\ra:\Omega\times\Omega\ra\Omega$ are maps such that, for any $\alpha,\beta,\gamma\in\Omega:$
\begin{align}
(\alpha\la\beta)\la\gamma&=\alpha\la(\beta\la\gamma),\notag\\
(\alpha\ra\beta)\la\gamma&=\alpha\ra(\beta\la\gamma),\notag\\
(\alpha\ra\beta)\ra\gamma&=\alpha\ra(\beta\ra \gamma).
\end{align}
\end{defn}
\begin{remark}
Any diassociative semigroup is a duplicial semigroup, but the converse is not true: the two properties
$$\alpha\la(\beta\la\gamma)=\alpha\la(\beta\ra\gamma)\hbox{ and }(\alpha\la\beta)\ra\gamma=(\alpha\ra\beta)\ra\gamma$$
are always verified in a diassociative semigroup, but are not required in a duplicial semigroup.
\end{remark}
\ignore{
\begin{exam}~\cite{Foi20}
\begin{enumerate}
\item If $(\Omega, \circledast )$ is an associative semigroup, then $(\Omega, \circledast , \circledast )$ is a diassociative semigroup.
\item Let $\Omega$ be a set. For $\alpha,\beta \in \Omega,$ let
\begin{align*}
&\alpha \la \beta=\alpha,&\alpha \rightarrow \beta=\beta.
\end{align*}
Then $(\Omega,\la,\ra)$ is a diassociative semigroup denoted by $\DS(\Omega)$.
\end{enumerate}
\end{exam}
}
\begin{defn}\mlabel{defn:paredup}
Let $\Omega$ be a duplicial semigroup with two products $\la$ and $\ra$. {\bf A two-parameter $\Omega$-duplicial algebra} is a family $(A,(\prec_{\alpha,\beta},\succ_{\alpha,\beta})_{\alpha,\beta\in\Omega})$ where $A$ is a vector space and
$$\prec_{\alpha,\beta},\succ_{\alpha,\beta}:A\ot A\ra A$$
are bilinear binary products such that for any $x,y,z\in A$, for any $\alpha,\beta\in\Omega,$
\begin{align}
(x\prec_{\alpha,\beta}y)\prec_{\alpha\la\beta,\gamma}z
&=x\prec_{\alpha,\beta\la\gamma}(y\prec_{\beta,\gamma} z),\mlabel{eq:duprec}\\
(x\succ_{\alpha,\beta}y)\prec_{\alpha\ra\beta,\gamma}z&=x\succ_{\alpha,\beta\la\gamma}
(y\prec_{\beta,\gamma}z),\mlabel{eq:duprsu}\\
x\succ_{\alpha,\beta\ra\gamma}(y\succ_{\beta,\gamma}z)
&=(x\succ_{\alpha,\beta}y\succ_{\alpha\ra\beta,\gamma}z.\mlabel{eq:dusucc}
\end{align}
\end{defn}

Two-parameter $\Omega$-duplicial algebras are related to duplicial algebras
and duplicial semigroups by the following proposition:
\begin{prop} \label{dupropequivalence}
Let $\Omega$ be a set with two binary operations $\la$ and $\ra$.
\begin{enumerate}
\item Let $A$ be a $\bfk$-vector space and let
$$\prec_{\alpha,\beta},\succ_{\alpha,\beta}:A\ot A\ra A$$
be two families of bilinear binary products indexed by $\Omega\times\Omega$. We define products $\prec$ and $\succ$ on the space $A\otimes \bfk\omega$ by:
\begin{eqnarray}
(x\otimes\alpha)\prec(y\otimes\beta)&=&(x\prec_{\alpha,\beta}y)\otimes (\alpha\leftarrow\beta),\label{grad-one}\\
(x\otimes\alpha)\succ(y\otimes\beta)&=&(x\succ_{\alpha,\beta}y)\otimes (\alpha\rightarrow\beta).\label{grad-two}
\end{eqnarray}
If $(A\otimes \bfk\omega,\prec,\succ)$ is a duplicial algebra, 
then \eqref{eq:duprec}, \eqref{eq:duprsu} and \eqref{eq:dusucc} hold.\label{dupropequivalence-a}
\item The following conditions are equivalent:
\begin{enumerate}
\item For any $\big(A,(\prec_{\alpha,\beta},\succ_{\alpha,\beta})
_{\alpha,\beta\in\Omega}\big)$ where $A$ is a $\bfk$-vector space and where \eqref{eq:duprec}, \eqref{eq:duprsu} and \eqref{eq:dusucc} hold, the vector space $A\otimes \bfk\Omega$ endowed with the binary operations $\prec$ and $\succ$ defined by
\eqref{grad-one} and \eqref{grad-two} is a duplicial algebra,
\item $(\Omega,\la,\ra)$ is a duplicial semigroup,
\item Any $\big(A,(\prec_{\alpha,\beta},\succ_{\alpha,\beta})
_{\alpha,\beta\in\Omega}\big)$ where $A$ is a $\bfk$-vector space and where \eqref{eq:prec}, \eqref{eq:prsu} and \eqref{eq:succ} hold is a two-parameter $\Omega$-duplicial algebra.
\end{enumerate}
\label{dupropequivalence-b}
\end{enumerate}
\end{prop}
\noindent The proof is similar to the proof of Proposition \ref{propequivalence} and left to the reader.
\begin{remark}
A duplicial semigroup is nothing but a duplicial algebra in the monoidal category of sets. Hence the duplicial algebra structure on $A\otimes\bfk \Omega$ together with the $\Omega$-grading yields a duplicial algebra structure on $\Omega$. This property has no equivalent in the dendriform case. It owes to the fact that the duplicial operad is a set operad, whereas the dendriform operad is a linear operad which is not reducible to a set operad. Another occurrence of this phenomenon will be described in greater detail in Section \ref{sect:prelie} devoted to two-parameter family pre-Lie algebras.
\end{remark}
\noindent Now we give a concrete example of two-parameter $\Omega$-duplicial algebra, which uses typed decorated planar binary trees \mcite{BHZ, ZGM}.
\begin{defn}\mlabel{defn:tdtree}
Let $X$ and $\Omega$ be two sets. An {\bf $X$-decorated $\Omega\times \Omega$-typed (abbreviated two-parameter typed decorated) planar binary tree}
is a triple $T = (T,\dec, \type)$, where
\begin{enumerate}
\item $T$ is a planar binary tree.
\item $\dec: V(T)\ra X$ is a map, where $V(T)$ stands for the set of internal vertices of $T$,
\item $\type: IE(T)\ra \Omega\times \Omega$ is a map, where $IE(T)$ stands for the set of internal edges of $T$.
\end{enumerate}
\end{defn}

\begin{exam}
Let $X$ and $\Omega$ be two sets. The typed decorated planar binary trees with three internal vertices are

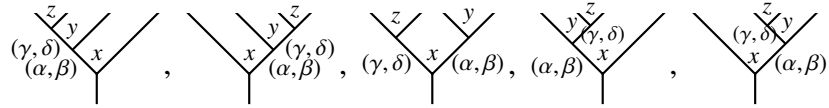
\[\XX[scale=1.6]{\xxr{-4}4\xxr{-7.5}{7.5}
\node at (-0.75,0.08) {$(\alpha,\beta)$};
\node at (-1.00,0.4) {$(\gamma,\delta)$};
\xxhu00{x} \xxhu[0.1]{-4}4{y} \xxhu[0.1]{-7.5}{7.5}{z}
}, \,
\XX[scale=1.6]{\xxl44\xxl{7.5}{7.5}
\node at (0.75,0.08) {$(\alpha,\beta)$};
\node at (1.00,0.4) {$(\gamma,\delta)$};
\xxhu00{x} \xxhu[0.1]44{y} \xxhu[0.1]{7.5}{7.5}{z}
}, \,
\XX[scale=1.6]{\xxr{-6}6\xxl66
\node at (-0.75,0.15) {$(\gamma,\delta)$};
\node at (0.75,0.15) {$(\alpha,\beta)$};
\xxhu00{x} \xxhu66{y} \xxhu{-6}6{z}
}, \,
\XX[scale=1.6]{\xxr{-5}5\xxl{-2}8
\node at (-0.75,0.1) {$(\alpha,\beta)$};
\node at (-0.00,0.65) {\tiny $(\gamma,\delta)$};
\xxhu00x
\xxhu[0.1]{-5}5{y} \xxhu[0.1]{-2}8{z}
}, \,
\XX[scale=1.6]{\xxl55\xxr28
\node at (0.75,0.15) {$(\alpha,\beta)$};
\node at (0.00,0.65) {\tiny $(\gamma,\delta)$};
\xxhu[0.1]00{x\,}
\xxhu[0.1]55{\,y} \xxhu[0.1]28{z}
}\]
with $x,y,z\in X$ and $(\alpha,\beta),(\gamma,\delta)\in\Omega\times\Omega$.
\end{exam}

Denote by $\DD$ the set of two-parameter typed decorated planar binary trees. For any $s\in\DD$ we denote by $\bar{s}$ the subjacent decorated tree, forgetting the types.
\begin{defn}\mlabel{defn:act}
Let $\Omega$ be a set. For $\alpha,\beta\in\Omega,$ first define
\begin{equation}
s\prec_{\alpha,\beta}t:=\bar{s}\prec\bar{t}+\text{ following types }
\mlabel{eq:dpre}
\end{equation}
which means grafting $t$ on $s$ at the rightmost leaf, and the types follow the rules below:
\begin{itemize}
\item the new edge is typed by the pair $(\alpha,\beta);$
\item any internal edge of $t$ has its type moved as follows:
$$(\omega,\tau)\mapsto (\omega,\tau\la\beta);$$
\item any internal edge of $s$ has its type moved as follows:
$$(\omega,\tau)\mapsto (\alpha\la\omega,\tau);$$
\item other edges keep their types unchanged.
\end{itemize}
Similarly, we second define
\begin{equation}
s\succ_{\alpha,\beta}t:=\bar{s}\succ\bar{t}+\text{ following types }
\mlabel{eq:dsuc}
\end{equation}
which means grafting $s$ on $t$ at the leftmost leaf, and the types follow the following rules:
\begin{itemize}
\item the new edge is typed by the pair $(\alpha,\beta);$
\item any internal edge of $t$ has its type moved as follows:
$$(\omega,\tau)\mapsto (\alpha\ra\omega,\tau);$$
\item any internal edge of $s$ has its type moved as follows:
$$(\omega,\tau)\mapsto (\omega,\tau\ra\beta);$$
\item other edges keep their types unchanged.
\end{itemize}
\end{defn}

\begin{exam}
Let $X$ and $\Omega$ be two sets. Let
$$s=\XX[scale=1.6]{\xxr{-6}6\xxl66
\node at (-0.9,0.15) {$(\alpha_1,\alpha_2)$};
\node at (0.85,0.15) {$(\beta_1,\beta_2)$};
\xxhu00{x} \xxhu66{z} \xxhu{-6}6{y}
},\text { and }\,
t=\XX[scale=1.6]{\xxl66
\node at (0.85,0.15) {$(\beta_1,\beta_2)$};
\xxhu00{m} \xxhu66{n} 
}.$$ Then
\begin{align*}
s\succ_{\alpha,\beta}t&=\XX[scale=1.6]{\xxr{-6}6\xxl66
\node at (-0.9,0.15) {$(\alpha_1,\alpha_2)$};
\node at (0.85,0.15) {$(\beta_1,\beta_2)$};
\xxhu00{x} \xxhu66{z} \xxhu{-6}6{y}
}\succ_{\alpha,\beta}\XX[scale=1.6]{\xxl66
\node at (0.85,0.15) {$(\beta_1,\beta_2)$};
\xxhu00{m} \xxhu66{n} 
}
=\XX[scale=2.0]{\xxr{-5}5\xxl{-2}8\xxr{-8}{8}\xxl66
\node at (-0.65,0.15) {$(\alpha,\beta)$};
\node at (-1.5,0.55) {$(\alpha_1,\alpha_2\ra\beta)$};
\node at (1.1,0.15) {$(\alpha\ra\beta_1,\beta_2)$};
\node (a) at (0.15,1.5) {$(\alpha_3,\alpha_4\ra\beta)$};
\path (-0.35,0.65)+(-30:1pt) coordinate (b);
\draw[->] (a) to[out=-90,in=-30] (b);
\xxhu00m
\xxhu[0.1]{-5}5{x} \xxhu[0.1]{-2}8{z}\xxhu[0.1]{-8}{8}{y}\xxhu{6}6{n}
},\\
s\prec_{\alpha,\beta}t&=\XX[scale=1.6]{\xxr{-6}6\xxl66
\node at (-0.9,0.15) {$(\alpha_1,\alpha_2)$};
\node at (0.85,0.15) {$(\beta_1,\beta_2)$};
\xxhu00{x} \xxhu66{z} \xxhu{-6}6{y}
}\prec_{\alpha,\beta}\XX[scale=1.6]{\xxl66
\node at (0.85,0.15) {$(\beta_1,\beta_2)$};
\xxhu00{m} \xxhu66{n} 
}=\XX[scale=2.5]{\xxl33\xxl{5.5}{5.5}\xxr{-7.5}{7.5}\xxl{8.5}{8.5}
\node at (-1.2,0.3) {$(\alpha_1,\alpha_2\la\beta)$};
\node at (0.7,0) {$(\alpha_3,\alpha_4\la\beta)$};
\node at (1.4,0.65) {$(\alpha\la\beta_1,\beta_2)$};
\node at (0.8,0.40) {$(\alpha,\beta)$};
\xxhu00{x} \xxhu[0.1]33{z} \xxhu[0.1]{5.5}{5.5}{m} \xxhu[0.1]{-7}7{y}\xxhu[0.1]{8.5}{8.5}{n}
}
\end{align*}
\end{exam}

\begin{prop}\mlabel{prop:dup}
Let $X$ and $\Omega$ be two sets. The pair $\big(\DD,(\prec_{\alpha,\beta},\succ_{\alpha,\beta})_{\alpha,\beta\in\Omega}\big)$ is a two-parameter $\Omega$-duplicial algebra.
\end{prop}

\begin{proof}
For $s,t,u\in\DD$ and $\alpha,\beta,\gamma\in\Omega$, we first prove Eq.~(\mref{eq:duprec}). Let us look at the right hand side of Eq.~(\mref{eq:duprec}), that is, $s\prec_{\alpha,\beta\la\gamma}(t\prec_{\beta,\gamma}u)$. We divide the procedure into two steps.
\begin{itemize}
\item \textsl{First step}: we deal with $t\prec_{\beta,\gamma}u$, we have the new edge typed by $(\beta,\gamma)$; the edges of $u$ have their types $(\omega,\tau)$ changed into $(\beta\la\omega,\tau)$; the edges of $t$ have their types $(\omega,\tau)$ changed into $(\omega,\tau\la\omega).$
\item\textsl {Second step}: we deal with $s\prec_{\alpha,\beta\la\gamma}(t\prec_{\beta,\gamma}u)$, which means grafting $t\prec_{\beta,\gamma}u$ on the rightmost leaf of $s$. The new edge has its type $(\alpha,\beta\la\gamma)$; the new edge of $t\prec_{\beta,\gamma}u$ produced in the first step has its type $(\beta,\gamma)$ changed into $(\alpha\la\beta,\gamma)$; the edges of $u$ have their types $(\beta\la\omega,\tau)$ changed into $\big(\alpha\la(\beta\la\omega),\tau\big)$; the edges of $t$ have their types $(\omega,\tau\la\gamma)$ changed into $(\alpha\la\omega,\tau\la\gamma)$; the edges of $s$ have their types $(\omega,\tau)$ changed into $\bigl(\omega,\tau\la(\beta\la\gamma)\bigr).$
\end{itemize}
Let us now look at the left hand side of Eq.~(\mref{eq:duprec}), that is, $(s\prec_{\alpha,\beta}t)\prec_{\alpha\la\beta,\gamma} u$. We also divide into the procedure two steps.
\begin{itemize}
\item \textsl{First step}: we deal with $s\prec_{\alpha,\beta}t$: we graft $t$ on $s$, and the new edge typed by $(\alpha,\beta)$; the edges of $t$ have their types $(\omega,\tau)$ changed into $(\alpha\la\omega,\tau)$; the edges of $s$ have their types $(\omega,\tau)$ changed into $(\omega,\tau\la\beta).$
\item\textsl{Second step}: we deal with $(s\prec_{\alpha,\beta}t)\prec_{\alpha\la\beta,\gamma}u$. The new edge typed by $(\alpha\la\beta,\gamma)$; the new edge of $s\prec_{\alpha,\beta}t$ has its type $(\alpha,\beta)$ changed into $(\alpha,\beta\la\gamma)$; the edges of $s$ have their types $(\omega,\tau\la\beta)$ changed into $(\omega,(\tau\la\beta)\la\gamma)$; the edges of $t$ have their type $(\alpha\la\omega,\tau)$ changed into $(\alpha\la\omega,\tau\la\gamma)$; the edges of $u$ have their types $(\omega,\tau)$ changed into $((\alpha\la\beta)\la\omega,\tau).$
\end{itemize}
Comparing both sides and using the duplicial semigroup axioms proves Equation (\mref{eq:duprec}).\\

\noindent Second, we prove Equation~\eqref{eq:duprsu}. We use a table for comparison.
$$
\begin{tabular}{c|c}
  \hline
  $s\succ_{\alpha,\beta\la \gamma}(t\prec_{\beta,\gamma}u)$ & $(s\succ_{\alpha,\beta}t)\prec_{\alpha\ra\beta,\gamma}u$  \\
  \hline

  \text{ the first step}: $t\prec_{\beta,\gamma}u$ & \text{ the first step}: $s\succ_{\alpha,\beta}t$  \\
  \hline

  \text{ new edge typed by} $(\beta,\gamma)$ & \text{ new edge typed by} $(\alpha,\beta)$  \\
  \hline

  \text{ the edges of $u$}   & \text{ the edges of $s$ }  \\

   $(\omega,\tau)\mapsto(\beta\la\omega,\tau)$ & $(\omega,\tau)\mapsto(\omega,\tau\ra\beta)$\\
  \hline
  \text{ the edges of $t$ } & \text{ the edges of $t$ }  \\

   $(\omega,\tau)\mapsto(\omega,\tau\la\gamma)$ &  $(\omega,\tau)\mapsto(\alpha\ra\omega,\tau)$\\
 \hline
  \text{ the second step:} $s\succ_{\alpha,\beta\la\gamma}(t\prec_{\beta,\gamma}u)$ & \text{ the second step:} $(s\succ_{\alpha,\beta}t)\prec_{\alpha\ra\beta,\gamma}u$  \\

 \hline
  \text{ the new edge typed by $(\alpha,\beta\la\gamma)$} & \text{ the new edge typed by $(\alpha\ra\beta,\gamma)$}  \\
  \hline
  \text{ the new edge of $t\prec_{\beta,\gamma}u$ } & \text{ the new edge of $s\succ_{\alpha,\beta}t$ }  \\

   $(\omega,\gamma)\mapsto(\alpha\ra\beta,\gamma)$ & $(\alpha,\beta)\mapsto(\alpha,\gamma\la\gamma)$\\
  \hline
  \text{ the edges of $t$ } & \text{ the edges of $t$ }  \\

   $(\omega,\tau\la\gamma)\mapsto(\alpha\ra\omega,\tau\la\gamma)$ & $(\alpha\ra\omega,\tau)\mapsto(\alpha\ra\omega,\tau\la\gamma)$\\
  \hline
  \text{ the edges of $u$}  & \text{ the edges of $u$ }  \\
   $(\beta\la\omega,\tau)\mapsto\big(\alpha\ra(\beta\la\omega),\tau\big)$ & $(\omega,\tau)\mapsto\big((\alpha\la\beta)\la\omega,\tau\big)$\\
  \hline
  \text{ the edges of $s$ } & \text{ the edges of $s$ }  \\
  $(\omega,\tau)\mapsto(\omega,\tau\ra(\beta\la\gamma))$ & $(\omega,\tau\ra\beta)\mapsto(\omega,(\tau\ra\beta)\la\gamma)$\\
  \hline
\end{tabular}
$$
So both the left hand side and right hand side coincide.\\

\noindent Last, Eq.~(\mref{eq:dusucc}) can be proved similarly to Eq.~(\mref{eq:duprec}). Details are left to the reader.
\end{proof}
\subsection{Free one-parameter $\Omega$-duplicial algebras}
We give a general definition of one-parameter $\Omega$-duplicial algebras in the spirit of \cite{Foi20}.
\ignore{
Similar to Example 1.2 of Definition 1 by the first author~\cite{Foi20}, we give the following definition.
\begin{defn}\label{defn:sus}
Let $\Omega$ be a set. For $\alpha,\beta\in\Omega,$ let
\[\alpha\la\beta=\alpha,\,\alpha\ra\beta=\beta.\]
Then $(\Omega,\la,\ra)$ is a duplicial semigroup, denote by $\DUS(\Omega)$.
\end{defn}
}
\begin{defn}\label{defn:edsu}
An extended duplicial semigroup (briefly, EDuS) is a family $(\Omega,\la,\ra,\lhd,\rhd)$, where $\Omega$ is a set and $\la,\ra,\lhd,\rhd:\Omega\times\Omega\ra\Omega$ are maps such that:
\begin{enumerate}
\item $(\Omega,\la,\ra)$ is a duplicial semigroup.
\item For any $\alpha,\beta,\gamma\in\Omega$,
\begin{align}
\label{eq:ext0} \alpha\rhd (\beta \leftarrow \gamma)&=\alpha \rhd \beta,\\
\label{eq:ext00} (\alpha \rightarrow \beta)\lhd \gamma&=\beta \lhd \gamma,\\
\label{eq:ext1}(\alpha\lhd\beta)\la\Big((\alpha\la\beta)\lhd\gamma\Big)
&=\alpha\lhd(\beta\la\gamma),\\
\label{eq:ext2}(\alpha\lhd\beta)\lhd\Big((\alpha\la\beta)\lhd\gamma\Big)
&=\beta\lhd\gamma,\\
\label{eq:ext3}\Big(\alpha\rhd(\beta\ra\gamma)\Big)\rhd(\beta\rhd\gamma)
&=\alpha\rhd\beta,\\
\label{eq:ext4}\Big(\alpha\rhd(\beta\ra\gamma)\Big)\ra(\beta\rhd\gamma)
&=(\alpha\ra\beta)\rhd\gamma.
\end{align}
\end{enumerate}
\end{defn}

\begin{remark}
Any extended diassociative semigroup is an extended duplicial semigroup, but among the 10 axioms describing the compatibility between the arrows and the triangles in an extended diassociative semigroup (numbers 4 to 13 in \cite{Foi20}), only six of them survive in an EDuS (numbers 4, 5, 6, 7, 12 and 13).
\end{remark}
\ignore{
\noindent Similarly to Example 2.1 of Definition 2 by the first author~\cite{Foi20}, we give the following definition.
\begin{defn}\label{defn:eds}
Let $\Omega=(\Omega,\leftarrow,\rightarrow)$ be a duplicial semigroup. For $\alpha,\beta \in \Omega,$ we define two products on $\Omega$ by:
\begin{align*}
&\alpha \lhd \beta=\beta,&\alpha \rhd \beta=\alpha.
\end{align*}
Then $(\Omega,\leftarrow,\rightarrow,\lhd,\rhd)$ is an EDuS, denoted
by $\EDUS(\Omega,\leftarrow,\rightarrow)$. When $(\Omega,\leftarrow,\rightarrow)=\DUS(\Omega)$,
we shall simply write $\EDUS(\Omega)$.
\end{defn}
}
\begin{defn}\label{defn:dup}
Let $(\Omega,\leftarrow,\rightarrow,\lhd,\rhd)$ be an EDuS. A one-parameter $\Omega$-duplicial algebra is a family $\big(A,(\prec_\alpha)_{\alpha\in\Omega},(\succ_\alpha)_{\alpha\in\Omega}\big)$, where $A$ is a vector space and $\prec_\alpha,\succ_\alpha:A\ot A\ra A$ such that for any $x,y,z\in A$ and $\alpha,\beta\in\Omega,$
\begin{align}
\label{eq:du1}(x\prec_\alpha y)\prec_\beta z&=x\prec_{\alpha\la\beta}(y\prec_{\alpha\lhd\beta}z),\\
\label{eq:du2}x\succ_\alpha(y\prec_\beta z)&=(x\succ_\alpha y)\prec_\beta z,\\
\label{eq:du3}x\succ_\alpha(y\succ_\beta z)&=(x\succ_{\alpha\rhd\beta}y)\succ_{\alpha\ra\beta}z.
\end{align}
\end{defn}
\ignore{
\begin{remark}
In fact, $\Omega$-duplicial algebras are closely related to $\Omega$-dendriform algebras by the first author~\cite{Foi20}, where Eq.~(\ref{eq:du2}) is the same than Eq.~(42), but Eqs.~(\ref{eq:du1}) and~(\ref{eq:du3}) are precisely half of Eq.~(41) and~(43) in~\cite{Foi20}.
\end{remark}
}
Now we describe free one-parameter $\Omega$-duplicial algebras in terms of planar binary trees typed by $\Omega$, that is, for which each internal edge is typed by single element of $\Omega$. The set of $\Omega$-typed $X$-decorated planar binary trees is denoted by $\mathbf T(X,\Omega)$. We denote by $\mathbf T^+(X,\Omega)$ the set of $\Omega$-typed $X$-decorated planar binary trees different from the trivial tree $|$. For any $n\geq 0$, the set of $\Omega$-typed $X$-decorated planar binary trees with $n$ internal vertices (and $n+1$ leaves) is denoted by $\mathbf T_n(X,\Omega)$. So we have
\begin{align*}
\mathbf T(X,\Omega)&=\bigsqcup_{n\geqslant 0} \mathbf T_n(X,\Omega),&
\mathbf T^+(X,\Omega)&=\bigsqcup_{n\geqslant 1} \mathbf T_n(X,\Omega).
\end{align*}
For example,
\begin{align*}
\mathbf T_0(X,\Omega)&=\{|\},\ \
\mathbf T_1(X,\Omega)=\left\{\stree x\Bigm|x\in X
\right\},\ \
\mathbf T_2(X,\Omega)=\left\{
\XX{\xxr{-5}5
\node at (-0.4,0) {$\alpha$};
\xxhu00x \xxhu{-5}5y
}, \,
\XX{\xxl55
\node at (0.4,0) {$\alpha$};
\xxhu00x \xxhu55y
}\Bigm|x,y\in X,\alpha\in\Omega
\right\},\\
\mathbf T_3(X,\Omega)&=\left\{
\XX[scale=1.6]{\xxr{-4}4\xxr{-7.5}{7.5}
\node at (-0.35,0.1) {$\alpha$};
\node at (-0.75,0.4) {$\beta$};
\xxhu00{x} \xxhu[0.1]{-4}4{y} \xxhu[0.1]{-7.5}{7.5}{z}
}, \,
\XX[scale=1.6]{\xxl44\xxl{7.5}{7.5}
\node at (0.35,0.1) {$\alpha$};
\node at (0.7,0.4) {$\beta$};
\xxhu00{x} \xxhu[0.1]44{y} \xxhu[0.1]{7.5}{7.5}{z}
}, \,
\XX[scale=1.6]{\xxr{-6}6\xxl66
\node at (-0.5,0.15) {$\beta$};
\node at (0.45,0.15) {$\alpha$};
\xxhu00{x} \xxhu66{y} \xxhu{-6}6{z}
}, \,
\XX[scale=1.6]{\xxr{-5}5\xxl{-2}8
\node at (-0.4,0.1) {$\alpha$};
\node at (-0.2,0.55) {\tiny $\beta$};
\xxhu00x
\xxhu[0.1]{-5}5{y} \xxhu[0.1]{-2}8{z}
}, \,
\XX[scale=1.6]{\xxl55\xxr28
\node at (0.45,0.15) {$\alpha$};
\node at (0.22,0.5) {\tiny $\beta$};
\xxhu[0.1]00{x\,}
\xxhu[0.1]55{\,y} \xxhu[0.1]28{z}
},\ldots \Bigg|\,x,y,z\in X,\alpha,\beta\in\Omega\right\}.
\end{align*}

The {\bf depth} $\dep{(T)}$ of a rooted tree $T$ is the maximal length of linear chains of vertices from the root to the leaves of the tree.
For example,
$$\dep\biggl(\stree x\biggr) = 1\, \text{ and } \, \dep\biggl(\XX{\xxr{-5}5
\node at (-0.4,0) {$\alpha$};
\xxhu00x \xxhu{-5}5y
}\biggr) = 2.$$

\begin{defn}\label{defn:graft}
Let $T_1,T_2\in \mathbf T(X,\Omega)$, and $\alpha,\beta\in \Omega$. We denote by
$\displaystyle T_1 \Ve x,\alpha,\beta; T_2$ the tree $T\in \mathbf T(X,\Omega)$ obtained by grafting $T_1$ and $T_2$
on a common root. If $T_1\neq |$, the type of the internal edge between the root of $T$ and the root of $T_1$ is  $\alpha$.
If $T_2\neq |$, the type of  internal edge between the root of $T$ and the root of $T_2$ is $\beta$. We also decorate the new vertex by $x,x\in X.$
\end{defn}

\begin{remark}
Note that any element $T\in \mathbf T_n(X,\Omega)$, with $n\geq 1$, can be written under the form
\[T=T_1\Ve x,\alpha,\beta; T_2,\]
with $T_1,T_2\in \mathbf T(X,\Omega)$, $x\in X$ and $\alpha,\beta \in \Omega$. This writing is unique except if
$T_1=\mid$ or $T_2=\mid$: in this case, one can change arbitrarily $\alpha$ or $\beta$.
In order to solve this notational problem, we add an element denoted by $1$ to $\Omega$
and we shall always assume that if $T_1=\mid$, then $\alpha=1$;
if $T_2=\mid$, then $\beta=1$.
\end{remark}

\begin{defn}\label{def:freeDF}
Let $\Omega$ be a set with four products $\leftarrow, \rightarrow, \lhd, \rhd$.
We define binary operations $(\prec_\alpha,\succ_\alpha)_{\alpha\in\Omega}$ on $\bfk\mathbf T^+(X,\Omega)$ recursively on $\dep(T)+\dep(U)$ by
\begin{enumerate}
\item $ |\prec_\omega T:=T\succ_\omega|:=T$ for $\omega\in\Omega$ and $T\in \mathbf T^+(X,\Omega)$.\mlabel{it:intial1}
\item For $T=T_1\Ve x,\alpha_1,\alpha_2; T_2$ and $U=U_1\Ve y,\beta_1,\beta_2; U_2,$ define
\begin{align}
T\prec_\omega U:=& \ T_1\Ve x,\alpha_1,\alpha_2\la\omega; (T_2\prec_{\alpha_2\lhd\omega} U),\mlabel{eq:recur1}\\
T\succ_\omega U:=& \ (T\succ_{\omega\rhd\beta_1} U_1)\Ve y,\omega\ra\beta_1,\beta_2; U_2,\,\text{ where }\, \omega \in \Omega.\mlabel{eq:recur2}
\end{align}
\end{enumerate}
\mlabel{defn:recur}
\end{defn}
In the following, we employ the convention that
\begin{equation}
\label{eq:conv}\omega\rhd 1=1\lhd \omega=\omega\, \text{ and }\,\omega\ra 1=1\la\omega=\omega\, \text{ for }\,\omega\in\Omega.
\end{equation}

\begin{exam}\label{exam:dup}
Let $T=\stree x$ and $U=\stree y$ with $x,y\in X$. For $\omega\in \Omega,$ we have
\begin{align*}
T\prec_\omega U&=\stree x\prec_\omega\stree y=\Big(|\Ve x,1,1;|\Big)\prec_\omega \stree y\\
&=|\Ve x,1,1\la\omega;\Big(|\prec_{1\lhd\omega}\stree y\Big)\\
&=|\Ve x,1,\omega;\stree y
=\XX{\xxl55
\node at (0.4,0) {$\omega$};
\xxhu00x \xxhu55y
},\\
T\succ_\omega U&=\stree x\succ_\omega\stree y=\stree x\succ_\omega\Big(|\Ve y,1,1;|\Big)\\
&=\Big(\stree x\succ_{\omega\rhd 1}|\Big)\Ve y,\omega\ra 1,1;|\\
&=\stree x\Ve y,\omega,1;|
=\XX{\xxr{-5}5
\node at (-0.45,0) {$\omega$};
\xxhu00y \xxhu{-5}5{x}
}.
\end{align*}
\end{exam}

\begin{prop}\mlabel{prop:dup2}
Let $X$ be a set and let $(\Omega,\leftarrow,\rightarrow,\lhd,\rhd)$ be an EDuS.
Then $\big(\bfk\mathbf T^+(X,\Omega), (\prec_\omega,\succ_\omega)_{\omega\in\Omega}\big)$ is an $\Omega$-duplicial algebra.
\end{prop}

\begin{proof}
Let $$T=T_1\Ve x,\alpha_1,\alpha_2; T_2, U=U_1\Ve y,\beta_1,\beta_2; U_2, W=W_1\Ve z,\gamma_1,\gamma_2; W_2\in \bfk\mathbf T^+(X,\Omega).$$
 Then we apply induction on $\dep(T)+\dep(U)+\dep(W)\geq 3.$ For the initial step $\dep(T)+\dep(U)+\dep(W)=3$, we have $$T=\stree x, U=\stree y\,\text{ and }\, W=\stree z$$ and so
 \begin{align*}
 (T\prec_\alpha U)\prec_\beta W&=\Big(\stree x\prec_\alpha \stree y\Big)\prec_\beta\stree z\\
 &=\XX{\xxl55
\node at (0.4,0) {$\alpha$};
\xxhu00x \xxhu55y
}\prec_\beta \stree z\quad(\text{ by Example~\ref{exam:dup} })\\
&=\Big(|\Ve x,1,\alpha;\stree y\Big)\prec_\beta\stree z\\
&=|\Ve x,1,\alpha\la\beta;\Big(\stree y\prec_{\alpha\lhd\beta}\stree z\Big)\quad(\text{by Eq.~(\ref{eq:recur1})})\\
&=|\Ve x,1,\alpha\la\beta;\XX{\xxl55
\node at (0.75,0) {$\alpha\lhd\beta$};
\xxhu00y \xxhu55z
}\quad(\text{ by Example~\ref{exam:dup} })\\
&=\stree x\prec_{\alpha\la\beta}\XX{\xxl55
\node at (0.75,0) {$\alpha\lhd\beta$};
\xxhu00y \xxhu55z
}\quad(\text{by Eq.~(\ref{eq:recur1})})\\
&=\stree x\prec_{\alpha\la\beta}\Big(\stree y\prec_{\alpha\lhd\beta}\stree z\Big)\\
&=T\prec_{\alpha\la\beta}(U\prec_{\alpha\lhd\beta}W),
 \end{align*}
 verifying Eq.~(\mref{eq:du1}). Next,
\begin{align*}
T\succ_\alpha(U\prec_\beta W)&=\stree x\succ_\alpha\Big(\stree y\prec_\beta\stree z\Big)\\
&=\stree x\succ_\alpha \XX{\xxl55
\node at (0.4,0) {$\beta$};
\xxhu00y \xxhu55z
}\quad(\text{ by Example~\ref{exam:dup} })\\
&=\stree x\succ_\alpha\Big(|\Ve y,1,\beta;\stree z\Big)\\
&=\Big(\stree x\succ_{\alpha\rhd 1}|\Big)\Ve y,\alpha\ra 1,\beta;\stree z\quad(\text{by Eq.~(\ref{eq:recur2})})\\
&=\stree x\Ve y,\alpha,\beta;\stree z\quad(\text{by Eq.~(\ref{eq:conv})})\\
&=\stree x\Ve y,\alpha,1\la\beta;\Big(|\prec_{1\lhd\beta}\stree z\Big)\quad(\text{by Item~\ref{it:intial1} of Definition~\ref{def:freeDF} and Eq.~(\ref{eq:conv})})\\
&=\Big(\stree x\Ve y,\alpha,1;|\Big)\prec_\beta\stree z\quad(\text{by Eq.~(\ref{eq:recur1})})\\
&=\XX{\xxr{-5}5
\node at (-0.45,0) {$\alpha$};
\xxhu00y \xxhu{-5}5{x}
}\prec_\beta \stree z
=\Big(\stree x\succ_\alpha\stree y\Big)\prec_\beta\stree z\\
&=(T\succ_\alpha U)\prec_\beta W,
\end{align*}
verifying Eq.~(\mref{eq:du2}). Finally,
\begin{align*}
T\succ_\alpha(U\succ_\beta W)&=\stree x\succ_\alpha\Big(\stree y\succ_\beta \stree z\Big)\\
&=\stree x\succ_\alpha \XX{\xxr{-5}5
\node at (-0.45,0) {$\beta$};
\xxhu00z \xxhu{-5}5{y}
}\quad(\text{ by Example~\ref{exam:dup} })\\
&=\stree x\succ_\alpha\Big(\stree y\Ve z,\beta,1;|\Big)\\
&=\Big(\stree x\succ_{\alpha\rhd\beta}\stree y\Big)\Ve z,\alpha\ra\beta,1;|\quad(\text{by Eq.~(\ref{eq:recur2})})\\
&=\XX{\xxr{-5}5
\node at (-1,0) {$\alpha\rhd\beta$};
\xxhu00y \xxhu{-5}5{x}}\Ve z,\alpha\ra\beta,1;|\quad(\text{ by Example~\ref{exam:dup} })\\
&=\Big(\XX{\xxr{-5}5
\node at (-1,0) {$\alpha\rhd\beta$};
\xxhu00y \xxhu{-5}5{x}
}\succ_{(\alpha\ra\beta)\rhd 1}|\Big)\Ve z,(\alpha\ra\beta)\ra 1,1;|\quad(\text{by Eq.~(\ref{eq:conv})})\\
&=\XX{\xxr{-5}5
\node at (-1,0) {$\alpha\rhd\beta$};
\xxhu00y \xxhu{-5}5{x}
}\succ_{\alpha\ra\beta}\Big(|\Ve z,1,1;|\Big)
\quad(\text{by Item~\ref{it:intial1} of Definition~\ref{def:freeDF} and Eq.~(\ref{eq:conv})})\\
&=\XX{\xxr{-5}5
\node at (-1,0) {$\alpha\rhd\beta$};
\xxhu00y \xxhu{-5}5{x}
}\succ_{\alpha\ra\beta}\stree z\\
&=\Big(\stree x\succ_{\alpha\rhd\beta}\stree y\Big)\succ_{\alpha\ra\beta}\stree z\quad(\text{by Example~\ref{exam:dup} })\\
&=(T\succ_{\alpha\rhd \beta} U)\succ_{\alpha\ra\beta} W.
\end{align*}
This completes the proof of the initial step. For the induction step of $\dep(T)+\dep(U)+\dep(W)=k+1\geq 4$. First, we have
\begin{align*}
(T\prec_\alpha U)\prec_\beta W&=\Big((T_1\Ve x,\alpha_1,\alpha_2;T_2)\prec_\alpha U\Big)\prec_\beta W\\
&=\Big(T_1\Ve x,\alpha_1,\alpha_2\la\alpha);(T_2\prec_{\alpha_2\lhd \alpha } U)\Big)\prec_\beta W\quad(\text{by Eq.~(\ref{eq:recur1})})\\
&=T_1\Ve x,\alpha_1,(\alpha_2\la\alpha)\la\beta;\Big((T_2\prec_{\alpha_2\lhd\alpha}U)
\prec_{(\alpha_2\la\alpha)\lhd\beta} W\Big)\quad(\text{by Eq.~(\ref{eq:recur1})})\\
&=T_1\Ve x,\alpha_1,(\alpha_2\la\alpha)\la\beta;\Big(T_2\prec_{(\alpha_2\lhd\alpha)\la
\big((\alpha_2\la\alpha)\lhd\beta\big)}(U\prec_{(\alpha_2\lhd\alpha)\lhd
\big((\alpha_2\la\alpha)\lhd\beta\big)}W)\\
&\hspace{2cm}(\text{by the induction hypothesis})\\
&=T_1\Ve x,\alpha_1,\alpha_2\la(\alpha\la\beta);
\Big(T_2\prec_{\alpha_2\lhd(\alpha\la\beta)}(U\prec_{\alpha\lhd\beta}W)\Big)\\
&\hspace{2cm}(\text{by Eqs.~(\ref{eq:ext1}-\ref{eq:ext2}) and definition \eqref{defn:dus}})\\
&=(T_1\Ve x,\alpha_1,\alpha_2;T_2)\prec_{\alpha\la\beta}(U\prec_{\alpha\lhd\beta}W)
\quad(\text{by Eq.~(\ref{eq:recur1})})\\
&=T\prec_{\alpha\la\beta}(U\prec_{\alpha\lhd\beta}W).
\end{align*}
Second, we have
\begin{align*}
T\succ_\alpha(U\prec_\beta W)&=T\succ_\alpha\Big((U_1\Ve y,\beta_1,\beta_2;U_2)\prec_\beta W\Big)\\
&=T\succ_\alpha\Big(U_1\Ve y,\beta_1,\beta_2\la\beta;(U_2\prec_{\beta_2\lhd\beta}W)\Big)\quad(\text{by Eq.~(\ref{eq:recur2})})\\
&=(T\succ_{\alpha\rhd\beta_1}U_1)\Ve y,\alpha\ra\beta_1,\beta_2\la\beta;(U_2\prec_{\beta_2\lhd\beta}W)\\
&=\Big((T\succ_{\alpha\rhd\beta_1}U_1)\Ve y,\alpha\ra\beta_1,\beta_2;U_2\Big)\prec_\beta W\quad(\text{by Eq.~(\ref{eq:recur1})})\\
&=\Big(T\succ_\alpha(U_1\Ve y,\beta_1,\beta_2;U_2)\Big)\prec_\beta W\\
&=(T\succ_\alpha U)\prec_\beta W.
\end{align*}
Last, we have
\begin{align*}
T\succ_\alpha(U\succ_\beta W)&=T\succ_\alpha\Big(U\succ_\beta(W_1\Ve z,\gamma_1,\gamma_2;W_2)\Big)\\
&=T\succ_\alpha\Big((U\succ_{\beta\rhd\gamma_1}W_1)\Ve z,\beta\ra\gamma_1,\gamma_2;W_2\Big)\quad(\text{by Eq.~(\ref{eq:recur2})})\\
&=\Big(T\succ_{\alpha\rhd(\beta\ra\gamma_1)}(U\succ_{\beta\rhd\gamma_1}W_1)\Big)
\Ve z,\alpha\ra(\beta\ra\gamma_1),\gamma_2;W_2\quad(\text{by Eq.~(\ref{eq:recur2})})\\
&=\Big(T\succ_{(\alpha\rhd(\beta\ra\gamma_1))\rhd(\beta\rhd\gamma_1)}U\Big)
\succ_{(\alpha\rhd(\beta\ra\gamma_1))\ra(\beta\rhd\gamma_1)}W_1\Ve z,\alpha\ra(\beta\ra\gamma_1),\gamma_2;W_2\\
&\hspace{2cm}\text{(by the induction hypothesis)}\\
&=\Big((T\succ_{\alpha\rhd\beta}U)\succ_{(\alpha\ra\beta)\ra\gamma_1}W_1\Big)
\Ve z,(\alpha\ra\beta)\ra\gamma_1,\gamma_2;W_2\\
&\hspace{2cm}(\text{by Eqs.~(\ref{eq:ext3}-\ref{eq:ext4}) and Definition \eqref{defn:dus}})\\
&=(T\succ_{\alpha\rhd\beta}U)\succ_{\alpha\ra\beta}(W_1\Ve z,\gamma_1,\gamma_2;W_2)\\
&=(T\succ_{\alpha\rhd\beta}U)\succ_{\alpha\ra\beta}W.
\end{align*}
This completes the proof.
\end{proof}
\ignore{
\noindent There is a converse of Proposition~\ref{prop:dup}:

\begin{coro}\label{coro:dup}
Let $X$ be a set and let
$(\bfk\mathbf T^+(X,\Omega), (\prec_\omega,\succ_\omega)_{\omega\in\Omega})$ where the binary products $\prec_\omega$ and $\succ_\omega$ are given by Definition \ref{def:freeDF}. If this defines an $\Omega$-duplicial algebra, then $(\Omega,\leftarrow,\rightarrow,\lhd,\rhd)$ is an EDuS. Namely,
$(\bfk\mathbf T^+(X,\Omega), (\prec_\omega,\succ_\omega)_{\omega\in\Omega})$ is a $\Omega$-duplicial algebra, if and only if, $(\Omega,\leftarrow,\rightarrow,\lhd,\rhd)$ is an EDuS.
\end{coro}
\begin{proof}
All axioms of an EDuS have been used to prove Proposition~\ref{prop:dup}.
\end{proof}
\dominique{Actually... Where do you use axioms (21) and (22)?}
}
\begin{defn}
 Let $X$ be a set and let $(\Omega,\leftarrow,\rightarrow,\lhd,\rhd)$ be an EDuS.
 A free $\Omega$-duplicial algebra on $X$ is an $\Omega$-duplicial algebra $\big(D, (\prec_\omega,\succ_\omega)_{\omega\in\Omega}\big)$ together with a map $j: X\rightarrow D$ that satisfies the following universal property:
 for any $\Omega$-duplicial algebra $\big(D', (\prec'_\omega,\succ'_\omega)_{\omega\in\Omega}\big)$
and map $f: X\rightarrow D',$ there is a unique $\Omega$-duplicial algebra morphism $\bar{f}:
D\rightarrow D'$ such that $f=\bar{f}\circ j.$ The free $\Omega$-duplicial algebra on $X$ is unique up to isomorphism.
\end{defn}

\noindent Let $j: X \ra \bfk\mathbf T^+(X,\Omega)$ be the map defined by $j(x)=\stree x$ for $x\in X$.
\begin{theorem}\mlabel{thm:freedup}
Let $X$ be a set and let $(\Omega,\leftarrow,\rightarrow,\lhd,\rhd)$ be an EDuS.
Then $\big(\bfk\mathbf T^+(X,\Omega),(\prec_\omega,\succ_\omega)_{\omega\in\Omega}\big)$, together with the map $j$, is the free $\Omega$-duplicial algebra on $X$.
\end{theorem}

\begin{proof}
By Proposition~\ref{prop:dup}, we are left to show that $\big(\bfk\mathbf T^+(X,\Omega),(\prec_\omega,\succ_\omega)_{\omega\in\Omega}\big)$ satisfies the universal property. For this, let $\big(D,(\prec'_\omega,\succ'_\omega)_{\omega\in\Omega}\big)$ be an $\Omega$-duplicial algebra. 

\noindent Now let us define a linear map 
\[\bar{f}:\left\{\begin{array}{rcl}
\bfk\mathbf T^+(X,\Omega)&\rightarrow& D\\
T&\mapsto&\bar{f}(T)
\end{array}\right.\]
by induction on $\dep(T)\geq 1$.
Let us write $T=T_1\Ve x,\alpha_1,\alpha_2;T_2$ with $x\in X$ and $\alpha_1,\alpha_2\in\Omega$. For the initial step $\dep(T)=1$, we have $T=\stree x$ for some $x\in X$ and define
\begin{equation}
\bar{f}(T):=f(x).
\label{eq:init}
\end{equation}
We define $\bar{f}(T)$ by the induction on $\dep(T)=k+1\geq 2$.
Note that $T_1$ and $T_2$ can not be $|$ simultaneously and define
\begin{align}\label{eq:induc}
\bar{f}(T):=& \  \bar{f}(T_1\Ve x,\alpha_1,\alpha_2;T_2)\nonumber\\
:=& \ \left\{
\begin{array}{ll}
f(x)\prec'_{\alpha_2}\bar{f}(T_2), & \ \text{ if } T_1=|\neq T_2;\\
\bar{f}(T_1)\succ'_{\alpha_1} f(x), & \ \text{ if } T_1\neq |=T_2;\\
\bigl(\bar{f}(T_1)\succ'_{\alpha_1}f(x)\bigr)\prec'_{\alpha_2}\bar{f}(T_2), & \ \text{ if }T_1\neq |\neq T_2.
\end{array}
\right .
\end{align}

\noindent We are left to prove that $\bar{f}$ is a morphism of $\Omega$-duplicial  algebras:
 \begin{equation*}
\bar{f}(T\prec_\omega U)=\bar{f}(T)\prec'_\omega\bar{f}(U)\,\text{ and }\,
\bar{f}(T\succ_\omega U)=\bar{f}(T)\succ'_\omega\bar{f}(U),
\label{eq:comff}
\end{equation*}
in which we only prove the first equation by induction on $\dep(T)+\dep(U)\geq 2$,
as the proof of the second one is similar.
Write
$$T=_1\Ve x,\alpha_1,\alpha_2;T_2\,\text{ and }\, U=U_1\Ve y,\beta_1,\beta_2;U_2.$$
For the initial step $\dep(T)+\dep(U)=2$, we have $T=\stree x$ and $U=\stree y$ for some $x,y\in X$.
So we have
\begin{align*}
\bar{f}(T\prec_\omega U)&=\bar{f}\Big(\stree x\prec_\omega\stree y\Big)=\bar{f}\Big(\XX{\xxl55
\node at (0.4,0) {$\omega$};
\xxhu00x \xxhu55y
}\Big)\quad(\text{ by Example~\ref{exam:dup} })\\
&=\bar{f}\Big(|\Ve x,1,\omega;\stree y\Big)\\
&=f(x)\prec'_\omega\bar{f}\Big(\stree y\Big)\quad(\text{by Eq.~(\ref{eq:induc})})\\
&=\bar{f}\Big(\stree x\Big)\prec'_\omega\bar{f}\Big(\stree y\Big)\quad(\text{by Eq.~(\ref{eq:init})})\\
&=\bar{f}(T)\prec'_\omega\bar{f}(U).
\end{align*}
For the induction step of $\dep(T)+\dep(U)\geq 3,$ we have four cases to consider.

{\noindent \bf Case 1:} $T_1=|$ and $T_2=|.$ Then
\begin{align*}
\bar{f}(T\prec_\omega U)&=\bar{f}\Big((|\Ve x,1,1;|)\prec_\omega U\Big)\\
&=\bar{f}\Big(|\Ve x,1,1\la\omega;(|\prec_{1\lhd\omega}U)\Big)\quad(\text{by Eq.~(\ref{eq:recur1})})\\
&=\bar{f}(|\Ve x,1,\omega;U)=f(x)\prec'_\omega\bar{f}(U)\quad(\text{by Eq.~(\ref{eq:induc})})\\
&=\bar{f}(T)\prec'_\omega\bar{f}(U).
\end{align*}

{\noindent \bf Case 2:} $T_1=|$ and $T_2\neq |.$ Then
\begin{align*}
\bar{f}(T\prec_\omega U)&=\bar{f}\Big((|\Ve x,1,\alpha_2;T_2)\prec_\omega U\Big)\\
&=\bar{f}\Big(|\Ve x,1,\alpha_2\la\omega;(T_2\prec_{\alpha\lhd\omega}U\Big)\quad(\text{by Eq.~(\ref{eq:recur1})})\\
&=f(x)\prec'_{\alpha\la\omega}\bar{f}(T_2\prec_{\alpha_2\lhd\omega}U)\quad(\text{by Eq.~(\ref{eq:induc})})\\
&=f(x)\prec'_{\alpha\la\omega}\Big(\bar{f}(T_2)\prec'_{\alpha_2\lhd\omega}\bar{f}(U)\Big)
\quad(\text{by the induction hypothesis})\\
&=\Big(f(x)\prec'_{\alpha_2}\bar{f}(T_2)\Big)\prec'_\omega\bar{f}(U)\quad(\text{by Eq.~(\ref{eq:du1})})\\
&=\bar{f}(T)\prec'_\omega\bar{f}(U)\quad(\text{by Eq.~(\ref{eq:induc})}).
\end{align*}

{\noindent \bf Case 3:} $T_1\neq|$ and $T_2=|.$ This case is similar to Case 2.

{\noindent \bf Case 4:} $T_1\neq|$ and $T_2\neq |.$ Then
\begin{align*}
\bar{f}(T\prec_\omega U)&=\bar{f}\Big((T_1\Ve x,\alpha_1,\alpha_2;T_2)\prec_\omega U\Big)\\
&=\bar{f}\Big(T_1\Ve x,\alpha_1,\alpha_2\la\omega;(T_2\prec_{\alpha_2\lhd\omega} U)\Big)\quad(\text{by Eq.~(\ref{eq:recur1})})\\
&=\Big(\bar{f}(T_1)\succ'_{\alpha_1}f(x)\Big)
\prec'_{\alpha_2\la\omega}\bar{f}(T_2\prec_{\alpha_2\lhd\omega}U)\quad(\text{by Eq.~(\ref{eq:induc})})\\
&=\Big(\bar{f}(T_1)\succ'_{\alpha_1}f(x)\Big)
\prec'_{\alpha_2\la\omega}\Big(\bar{f}(T_1)\prec'_{\alpha_2\lhd\omega}\bar{f}(U)\Big)
\quad(\text{by the induction hypothesis})\\
&=\bar{f}(T_1)\succ'_{\alpha_1}\Big(f(x)\prec'_{\alpha_2\la\omega}
(\bar{f}(T_2)\prec'_{\alpha_2\lhd\omega}\bar{f}(U))\Big)\quad(\text{by Eq.~(\ref{eq:du2})})\\
&=\bar{f}(T_1)\succ'_{\alpha_1}
\Big((f(x)\prec'_{\alpha_2}\bar{f}(T_2))\prec'_\omega\bar{f}(U)\Big)\quad(\text{by Eq.~(\ref{eq:du1})})\\
&=\Big(\bar{f}(T_1)\succ'_{\alpha_1}(f(x)\prec'_{\alpha_2}\bar{f}(T_2)\Big)
\prec'_\omega\bar{f}(U)\quad(\text{by Eq.~(\ref{eq:du2})})\\
&=\Big(\Big(\bar{f}(T_1)\succ'_{\alpha_1}f(x)\Big)\prec'_{\alpha_2}\bar{f}(T_2)\Big)\prec'_\omega
\bar{f}(U)\quad(\text{by Eq.~(\ref{eq:du2})})\\
&=\bar{f}(T)\prec'_\omega\bar{f}(U)\quad(\text{by Eq.~(\ref{eq:induc})}).
\end{align*}

Let us prove the uniqueness of $\bar{f}.$ Let $\bar{g}$ be another morphism from $\bfk\mathbf T^+(X,\Omega)$ to $D$ such that $\bar{g}\Big(\stree x\Big)=f(x).$
First, for any $a\in D$ and $\omega\in\Omega$, we define
\begin{align*}
a\succ_\omega 1=1\prec_\omega a&=0,& 1\succ_\omega a=a\prec_\omega 1&=a.
\end{align*}
For any $T\neq |$, let $T=T_1\Ve x,\alpha_1,\alpha_2;T_2.$ In fact, the form $T=T_1\succ_{\alpha_1}\stree x\prec_{\alpha_2}T_2$ include all the above four cases. We define
\[\bar{g}(|)=1\]
and
\[\bar{g}(T)=\bar{g}\Big(T_1\succ_{\alpha_1}\stree x\prec_{\alpha_2}T_2\Big)=\bar{g}(T_1)\succ'_{\alpha_1}f(x)\prec'_{\alpha_2}\bar{g}(T_2).\]
So $\bar{f}=\bar{g}.$
This completes the proof.
\end{proof}

\begin{prop}
Let $\Omega$ be an EDuS. Then
$\big(\bfk\mathbf T^+(X,\Omega)\ot\bfk\Omega,\prec,\succ\big)$ is a duplicial algebra, if and only if,  $\big(\bfk\mathbf T^+(X,\Omega), (\prec_\omega,\succ_\omega)_{\omega\in\Omega}\big)$ is an $\Omega$-duplicial algebra, where
\begin{align*}
(x\ot\alpha)\prec(y\ot\beta):=&(x\prec_{\alpha\lhd\beta} y)\ot(\alpha\la\beta)\\
 (x\ot\alpha)\succ(y\ot\beta):=& (x\succ_{\alpha\rhd\beta} y)\ot(\alpha\ra\beta),\,
\text{ for }\, x,y\in \mathbf T^+(X,\Omega) \text{ and } \alpha,\beta\in \Omega.
\end{align*}
\end{prop}

\begin{proof}
For $x,y,z$ in the $\Omega$-duplicial algebra $\mathbf T^+(X,\Omega)$ and for $\alpha,\beta,\gamma\in\Omega,$ first, we prove
\begin{align*}
&\Big((x\ot\alpha)\prec(y\ot\beta)\Big)\prec(z\ot\gamma)\\
&=\Big(x\prec_{\alpha\lhd\beta}y\ot
(\alpha\la\beta)\Big)\prec(z\ot\gamma)\\
&=(x\prec_{\alpha\lhd\beta}y)\prec_{(\alpha\la\beta)\lhd\gamma}z
\ot\Big((\alpha\la\beta)\la\gamma\Big)\\
&=x\prec_{(\alpha\lhd\beta)\la\big((\alpha\la\beta)\lhd\gamma\big)}
(y\prec_{(\alpha\lhd\beta)\lhd\big((\alpha\la\beta)\lhd\gamma\big)}z)\ot
\Big((\alpha\la\beta)\la\gamma\Big)\quad(\text{by Eq.~(\ref{eq:du1})})\\
&=x\prec_{\alpha\lhd(\beta\la\gamma)}(y\prec_{\beta\lhd\gamma}z)
\ot\Big(\alpha\la(\beta\la\gamma)\Big)\quad(\text{by Eqs.~(\ref{eq:ext1}-\ref{eq:ext2}) and Definition \eqref{defn:dus}})\\
&=x\ot\alpha\prec\Big(y\prec_{\beta\lhd\gamma}z\ot(\beta\la\gamma)\Big)\\
&=(x\ot\alpha)\prec\Big((y\ot\beta)\prec(z\ot\gamma)\Big).
\end{align*}
Second, we have
\begin{align*}
&(x\ot\alpha)\succ\Big((y\ot\beta)\prec(z\ot\gamma)\Big)\\
&=(x\ot\alpha)\succ
\Big(y\prec_{\beta\lhd\gamma}z\ot(\beta\la\gamma)\Big)\\
&=x\succ_{\alpha\rhd(\beta\la\gamma)}(y\prec_{\beta\lhd\gamma}z)
\ot\Big(\alpha\ra(\beta\la\gamma)\Big)\quad(\text{by Eq.~(\ref{eq:du2})})\\
&=x\succ_{\alpha\rhd\beta}(y\prec_{(\alpha\ra\beta)\lhd\gamma}z)
\ot\Big((\alpha\ra\beta)\la\gamma\Big)\quad(\text{by Eqs.~(\ref{eq:ext0}-\ref{eq:ext00}) and Definition \eqref{defn:dus}})\\
&=\Big((x\succ_{\alpha\rhd\beta}y)\prec_{(\alpha\ra\beta)\lhd\gamma}z\Big)
\ot\Big((\alpha\ra\beta)\la\gamma\Big)\quad(\text{by Eq.~(\ref{eq:du2})})\\
&=\Big((x\succ_{\alpha\rhd\beta}y)\ot(\alpha\ra\beta)\Big)\prec(z\ot\gamma)\\
&=\Big((x\ot\alpha)\succ(y\ot\beta)\Big)\prec(z\ot\gamma).
\end{align*}
Finally, we have
\begin{align*}
&(x\ot\alpha)\succ\Big((y\ot\beta)\succ(z\ot\gamma)\Big)\\
&=x\ot\alpha\succ\Big(y\succ_{\beta\rhd\gamma}z\ot(\beta\ra\gamma)\Big)\\
&=x\succ_{\alpha\rhd(\beta\ra\gamma)}(y\succ_{\beta\rhd\gamma}z)
\ot\Big(\alpha\ra(\beta\ra\gamma)\Big)\\
&=\Big(x\succ_{\big(\alpha\rhd(\beta\ra\gamma)\big)\rhd(\beta\rhd\gamma)}y\Big)
\succ_{\big(\alpha\ra(\beta\ra\gamma)\big)\ra(\beta\rhd\gamma)}z
\ot\Big(\alpha\ra(\beta\ra\gamma)\Big)\quad(\text{by Eq.~(\ref{eq:du3})})\\
&=(x\succ_{\alpha\rhd\beta}y)\succ_{(\alpha\ra\beta)\rhd\gamma}z
\ot\Big((\alpha\ra\beta)\ra\gamma\Big)\quad(\text{by Eqs.~(\ref{eq:ext3}-\ref{eq:ext4}) and Definition \eqref{defn:dus}})\\
&=\Big(x\succ_{\alpha\rhd\beta}y\ot(\alpha\ra\beta)\Big)\succ(z\ot\gamma)\\
&=\Big((x\ot\alpha)\succ(y\ot\beta)\Big)\succ(z\ot\gamma).
\end{align*}
The converse comes from the fact that all axioms of an EDuS have been used in the proof.
\end{proof}
\subsection{On the operads of two-parameter duplical or dendriform algebras}

Let us assume that the parameter set $\Omega$ is finite, and let us denote its cardinality by $w$.
We denote by $\opdend$, respectively $\opdup$, the non-sigma operad of two-parameter dendriform, respectively duplicial, algebras.

\begin{prop}
For all $n\geqslant 1$, we put
\[r_n=\dimK\big(\opdend(n)\big)=\dimK\big(\opdup(n)\big),\]
and we consider
\[R(X)=\sum_{n=1}^\infty r_n X^n\in \mathbb{Q}[[X]].\]
Then 
\begin{align}
\label{EQ4}w^2(w-1)R^3+w(w X+2w-2)R^2+(2w X-1)R+X&=0.
\end{align}
\end{prop}

\begin{proof}
With a presentation by generators and relations of $\opdend$ and $\opdup$, it turns out that these operads
own a basis of planar binary trees the vertices of which are decorated by elements $\prec_{\alpha,\beta}$ or $\succ_{\alpha,\beta}$,
with $\alpha,\beta \in \Omega$, avoiding the trees of the form:
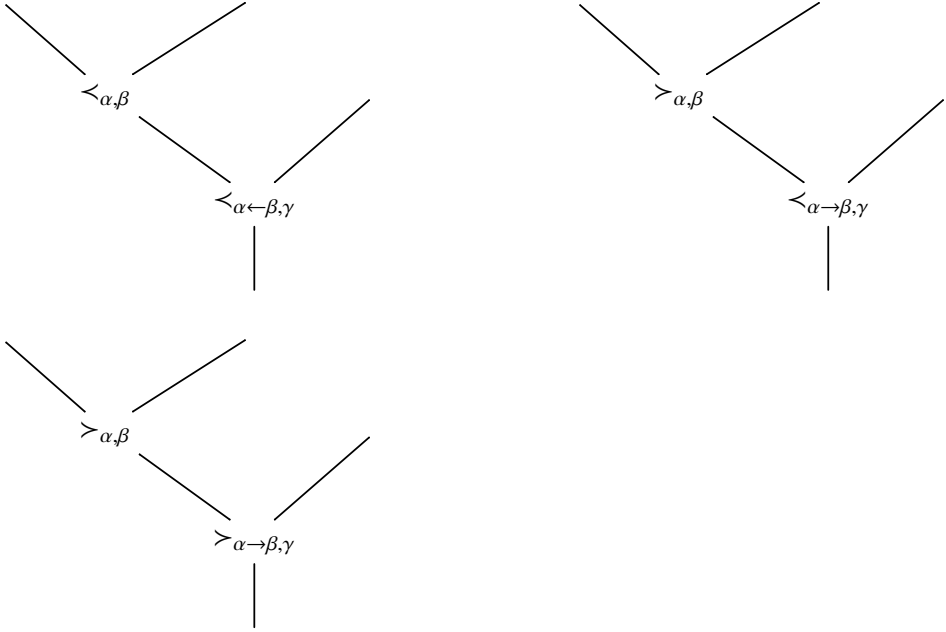
\begin{align*}
&\xymatrix{\ar@{-}[rd]&&\ar@{-}[ld]&&\\
&\prec_{\alpha,\beta} \ar@{-}[rd]&&\ar@{-}[ld]&\\
&&\prec_{\alpha \leftarrow \beta,\gamma}\ar@{-}[d]&&\\
&&&&}&
\xymatrix{\ar@{-}[rd]&&\ar@{-}[ld]&&\\
&\succ_{\alpha,\beta} \ar@{-}[rd]&&\ar@{-}[ld]&\\
&&\prec_{\alpha \rightarrow \beta,\gamma}\ar@{-}[d]&&\\
&&&&}\\
&\xymatrix{\ar@{-}[rd]&&\ar@{-}[ld]&&\\
&\succ_{\alpha,\beta} \ar@{-}[rd]&&\ar@{-}[ld]&\\
&&\succ_{\alpha \rightarrow \beta,\gamma}\ar@{-}[d]&&\\
&&&&}
\end{align*}
with $\alpha,\beta,\gamma \in \Omega$.
We denote by $R_\prec$ the formal series of such trees with root decorated by an element $\prec_{\alpha,\beta}$
and by  $R_\succ$ the formal series of such trees with root decorated by an element $\succ_{\alpha,\beta}$,
counted according to their number of leaves. Then:
\begin{align*}
R_\succ&=w^2(R_\prec+X)R+w(w-1)R_\succ R=w^2 R^2-w R_\succ R,\\
R_\prec&=w^2XR+w(w-1)(R_\succ+R_\prec)R=w^2R^2-w (R-X)R,\\
R&=X+R_\prec+R_\succ.
\end{align*}
We obtain that:
\begin{align*}
R_\succ&=\frac{w^2R^2}{1+w R},&
R_\prec&=w(w-1)R^2+w XR.
\end{align*}
Replacing in $R=R_\prec+R_\succ+X$, we obtain (\ref{EQ4}). \end{proof}

For example:
\begin{align*}
r_1&=1,\\
r_2&=2w^2,\\
r_3&=w^3(8w-3),\\
r_4&=2w^4(20w^2-15w+2),\\
r_5&=w^5(224w^3-252w^2+75w-5),\\
r_6&=2w^6(672w^4-1008w^3+476w^2-77w+3),\\
r_7&=w^7(8448w^5-15840w^4+10320w^3-2772w^2+280w-7),\\
r_8&=2w^8(27456w^6-61776w^5+51480w^4-19635w^3+3420w^2-234w+4).
\end{align*}

\begin{remark}
If $w=1$, one recovers duplicial and dendriform algebras, and $r_n(1)$ is the $n+1$ Catalan number $\cat_{n+1}$,
sequence A000108 of the OEIS \cite{Sloane}. The sequences $r_n(w)$ for $w=2$, $3$ or $4$ are not referenced (yet) in the OEIS.
\end{remark}

\begin{prop} 
Let $n \geqslant 1$.
\begin{enumerate}
\item $r_n$ is a polynomial in $\Z[w]$, of degree $2n-2$, and its leading coefficient is $2^{n-1} \cat_n$.
\item If $n\geqslant 2$, there exists a polynomial $t_n \in \Z[w]$, such that $r_n=w^n t_n$.
Moreover, $t_n(0)=(-1)^nn$.
\end{enumerate}
\end{prop}

\begin{proof}
By (\ref{EQ4}), if $n\geqslant 2$,
\begin{equation}
\label{EQ5}
r_n=w^2(w-1)\sum_{i+j+k=n} r_ir_jr_k+w^2\sum_{i+j=n-1}r_ir_j+w(2w-2)\sum_{i+j=n}r_ir_j
+2w r_{n-1}.
\end{equation}
Let us proceed by induction on $n$. The results are obvious if $n\leqslant 3$. Let us assume that $n\geqslant 4$
and the results at all ranks $<n$. By (\ref{EQ5}), obviously $r_n \in \Z[w]$. Moreover, by the induction hypothesis:
\begin{itemize}
\item The first term of (\ref{EQ5}) is of degree $\leqslant 3+2n-6=2n-3$.
\item The second term of (\ref{EQ5}) is of degree $\leqslant 2+2n-6=2n-4$.
\item The third term of (\ref{EQ5}) is of degree $\leqslant 2+2n-4=2n-2$; its coefficient of degree $2n-2$ is
\[2\sum_{i+j=n} 2^{i-1} \cat_i 2^{j-1}\cat_j=2^{n-1}\sum_{i+j=n}\cat_i=2^{n-1}\cat_n.\]
\item The fourth term of (\ref{EQ5}) is of degree $\leqslant 1+2n-4=2n-3$.
\end{itemize}
Hence, $r_n$ is of degree $2n-2$ and its leading coefficient is $2^{n-1}\cat_n$. Still by the induction hypothesis:
\begin{itemize}
\item For the first term of (\ref{EQ5}):
\begin{itemize}
\item If $i,j,k\geqslant 2$, then $w^2(w-1)r_i r_j r_k$ is a multiple of $w^{n+2}$.
\item If only one of $i,j,k$ is equal to $1$, then $w^2(w-1)r_i r_j r_k$ is a multiple of $w^{n+1}$.
\item If two of $i,j,k$ are equal to $1$, then the other one is equal to $n-2\geqslant 2$ and
$w^2(w-1)r_i r_j r_k$ is a multiple of $w^n$.
\end{itemize}
Hence, this first term is a multiple of $w^n$ and its contribution to the coefficent of $w^n$ is
\[-3(-1)^{n-2}(n-2).\]
\item For the second term of (\ref{EQ5}):
\begin{itemize}
\item If $i,j\geqslant 2$, then $w^2 r_i r_j$ is a multiple of $w^{n+1}$.
\item If one of $i$ or $j$ is equal to $1$, then the second one is $n-2\geqslant 2$ and $w^2 r_i r_j$ is a multiple of $w^n$.
\end{itemize}
Hence, this second term is a multiple of $w^n$ and its contribution to the coefficent of $w^n$ is
\[2(-1)^{n-2}(n-2).\]
\item For the third term of (\ref{EQ5}):
\begin{itemize}
\item If $i,j\geqslant 2$, then $w^2 r_i r_j$ is a multiple of $w^{n+1}$.
\item If one of $i$ or $j$ is equal to $1$, then the second one is $n-1\geqslant 2$ and $w(2w-2) r_i r_j$ 
is a multiple of $w^n$.
\item If $n$ is even, then any coefficent of $r_n$ is even.
\end{itemize}
Hence, this third term is a multiple of $w^n$ and its contribution to the coefficent of $w^n$ is
\[-2\times 2(-1)^{n-1}(n-1).\]
\item The last term of (\ref{EQ5}) is a multiple of $w^n$ and its contribution to the coefficent of $w^n$ is
\[ 2(-1)^{n-1}(n-1).\]
\end{itemize}
Finally, $r_n$ is a multiple of $w^n$ and the coefficient of $w^n$ in $r_n$ is
\[-3(-1)^n(n-2)+2(-1)^n(n-2)+4(-1)^n(n-1)-2(-1)^n(n-1)=(-1)^n n. \]
Let us assume that $n$ is even. Then, in $\Z/2\Z[w]$:
\[r_n\equiv w^2(w-1)\sum_{i+j+k=n} r_ir_jr_k+w^2\sum_{i+j=n-1}r_ir_j+0[2].\]
As $n$ is even, in the first term, one or three of $i,j,k$ are even, so $r_i r_j r_k\equiv 0[2]$;
in the second term, one of $i,j$ is even, so $r_i r_j\equiv 0[2]$. Finally, $r_n \equiv 0[2]$.  \end{proof}
\section{Reminders on operads and colored operads in the species formalism}\label{sect:operadic}
Colored operads are natural tools to be used in the description of algebraic structures on graded objects. We give a description of those in the colored species formalism, mainly following the presentation of \cite{DCH2019}. We also give a reminder of the more familiar monochromatic case, i.e. ordinary operads, and we decribe a pair $(\overline{\mathcal F}, \mathcal U)$ of adjoint functors from colored operads to monochromatic operads and vice-versa, along the lines of \cite{A2020}.
\subsection{Colored species}\label{par:colored-sp}
Let $\mathcal C$ be a bicomplete symmetric monoidal category, i.e. with small limits and colimits, which in particular implies the existence of products and coproducts indexed by an arbitrary set. For example the category of sets (the product given by cartesian product and the coproduct given by disjoint union), or the category of vector spaces over a field $\bfk$ (the product given by cartesian product and the coproduct being given by direct sum) \cite{Mac, AHS}. The unit for the monoidal product will be denoted by $\mathbf 1$, or $\mathbf 1_{\mathcal C}$ if the mention of the category must be precised.\\

Monoidal categories of $\Omega$-graded objects, where $\Omega$ is a semigroup, have been considered in \cite{A2020}. The symmetric monoidal structure is given by the Cauchy product, which uses the semigroup structure of $\Omega$ in an essential way. In absence of such a structure on our set $\Omega$, we must go further and consider multiple gradings. Let $\mathcal F_\Omega$ be the category of \textsl{$\Omega$-colored finite sets} defined as follows:
\begin{itemize}
\item objects are triples $(A,\underline\alpha,\omega)$ where $A$ is a finite set, $\omega\in\Omega$ (the \textsl{output color}) and $\underline\alpha:A\to\Omega$ is a list of elements of $\Omega$ indexed by $A$ (the \textsl{input colors}).
\item morphisms are given by bijective maps from $A$ onto $B$ together with re-indexing of colors: a morphism
$$\varphi:(A,\underline\alpha,\omega)\longrightarrow (B,\underline\beta,\zeta)$$
is given by an underlying bijective map $\overline\varphi:A\to B$ under the two conditions that $\omega=\zeta$ and $\underline\alpha=\underline\beta\circ\varphi$, otherwise there is no morphism from $(A,\underline\alpha,\omega)$ to $(B,\underline\beta,\zeta)$.
\end{itemize}
\begin{defn}
An $\Omega$-colored species $\mathcal P$ in the bicomplete monoidal category $\mathcal C$ is a contravariant functor $(A,\underline\alpha,\omega)\mapsto \mathcal P_{A,\underline\alpha,\omega}$ from $\mathcal F_\Omega$ to $\mathcal C$. The $\Omega$-colored species is \textsl{positive} if moreover $\mathcal P_{\hbox{\tiny \o},-,\omega}=0_{\mathcal C}$ for any $\omega\in\Omega$, where $0_{\mathcal C}$ is the initial object.
\end{defn}
This definition is borrowed from \cite[Definition 2.2]{DCH2019} which provides a slightly more general framework: $\Omega$-colored species correspond to $(\Omega,\Omega$)-collections therein. This can be straightforwardly extended to \textsl{$\Omega$-colored bi-species}, where several output colors are also allowed, to treat the case of colored ProPs and properads, but we shall not pursue this line of thought here.
\subsection{A brief summary of the monochromatic case}
The category $\mathcal F_\Omega$ boils down to the category $\mathcal F$ of finite sets with bijections when the set $\Omega$ of colors is reduced to one element. We recover then the usual notion of (contravariant) species \cite{J1234, AM2010, MM}. A $\mathcal C$-species is a contravariant functor from $\mathcal F$ into $\mathcal C$, where $\mathcal F$ is the category of finite sets with bijections as morphisms. We stick to \textsl{positive species}, i.e. species $\mathcal P$ such that $\mathcal P_{\hbox{\tiny \o}}=0_{\mathcal C}$, where $0_{\mathcal C}$ is the initial object of the monoidal category $\mathcal C$ \cite{MM}. We adopt M.~Mendez' definition of an operad in the species formalism:
\begin{defn} \cite[Definition 3.1]{MM}
An operad is a monoid in the category of positive species.
\end{defn}
Hence the operads considered here have no nullary operations. To be concrete, it is a positive species $\mathcal P$ together with partial compositions
$$\circ_b:\mathcal P_B\ot\mathcal P_C\longrightarrow\mathcal P_{B\sqcup_bC}$$
for any $b\in B$, where $B\sqcup_b C$ stands for $(B\setminus\{b\})\sqcup C$, subject to both sequential and parallel associativity axioms, which are stated as follows: for any finite sets $B,C,D$, for any $\alpha\in\mathcal P_B$, $\beta\in\mathcal P_C$ and $\beta',\gamma\in\mathcal P_D$ we have
$$\left\{
\begin{matrix}
\alpha\circ_b(\beta\circ_c\gamma)&=&(\alpha\circ_b\beta)\circ_c\gamma,\\
(\alpha\circ_b\beta)\circ_{b'}\beta'&=&(\alpha\circ_{b'}\beta')\circ_{b}\beta.
\end{matrix}
\right.
$$
\ignore{
Recall that a totally ordered $\mathcal C$-species is a contravariant functor from $\vec{\mathcal F}$ into $\mathcal C$, where $\vec{\mathcal F}$ is the category of totally ordered finite sets with increasing bijections as morphisms. A \textbf{non-sigma operad} is a monoid in the category of totally ordered species. To be concrete, it is a totally ordered species $\mathcal P$ together with partial compositions
$$\circ_b:\mathcal P_B\ot\mathcal P_C\longrightarrow\mathcal P_{B\sqcup_bC}$$
for any $b\in B$ subject to both sequential and parallel associativity axioms. Here, the total order on $B\sqcup_b C$ is given by:
\begin{itemize}
\item for any $b',b''\in B\setminus\{b\}$, one has $b'<b''$ if and only if $b'<b''$ in $B$,
\item for any $c',c''\in C$, one has $c'<c''$ if and only if $c'<c''$ in $C$,
\item for any $b'\in B\setminus\{b\}$ and $c'\in C$, one has $b'<c$ if and only if $b'<b$.
\end{itemize}
\noindent We will stick to \textsl{positive species}, i.e. species $\mathcal P$ such that $\mathcal P_{\hbox{\tiny \o}}=0_{\mathcal C}$, where $0_{\mathcal C}$ is the initial object of the monoidal category $\mathcal C$ \cite{MM}. Hence the operads considered here have no nullary operations.
}
\subsection{Colored operads}\label{par:colored}
In a colored operad, a partial composition is possible if and only if the output color of the second argument matches the color of the chosen input of the first. This is formalized as follows:
\begin{defn}
The substitution product of two positive $\Omega$-colored species is defined by
\begin{equation}
(\mathcal P\boxtimes \mathcal Q)_{A,\underline\alpha,\omega}:=\coprod_{\pi\hbox{ \tiny set partition of }A}
\, \coprod_{\underline\gamma:\pi\to\Omega}
\mathcal P_{\pi,\underline\gamma,\omega}\otimes\bigotimes_{B\in\pi}\mathcal Q_{B,\underline\alpha\srestr{B},\underline\gamma(B)}.
\end{equation}
\end{defn}

The substitution product $\boxtimes$ is also defined on morphisms and is associative, making the category of positive $\Omega$-colored species a (non-symmetric) monoidal category. The unit is the colored species $\mathbf 1$ defined by $\mathbf 1_{A,\underline\alpha,\omega}=\mathbf 1_{\mathcal C}$ if $|A|=1$ and $\underline\alpha=\omega$, and  $\mathbf 1_{A,\underline\alpha,\omega}=0$ otherwise. It can be written as
\begin{equation}
\mathbf 1=\prod_{\omega\in\Omega}\mathbf 1^\omega
\end{equation}
where $\mathbf 1^\omega$ is the colored species defined by $\mathbf 1^\omega_{A,\underline\alpha,\zeta}=\mathbf 1_{\mathcal C}$ if $|A|=1$ and $\underline\alpha=\zeta=\omega$, and  $\mathbf 1^\omega_{A,\underline\alpha,\zeta}=0$ otherwise. The colored species $\mathbf 1^\omega$ is sometimes slightly abusively called \textsl{unit of color $\omega$}.
\begin{defn}
A \textbf{colored operad} is a monoid in the monoidal category of positive $\Omega$-colored species endowed with the substitution product.
\end{defn}
\noindent Concretely, the global multiplication $\gamma:\mathcal P\boxtimes\mathcal P\to\mathcal P$ is declined into functorial partial compositions
\begin{equation}\label{pccol}
\circ_a:\mathcal P_{A,\underline\alpha,\omega}\otimes\mathcal P_{B,\underline\beta,\zeta}
\longrightarrow
\left\{
\begin{matrix}
\mathcal P_{A\sqcup_aB,\,\underline\alpha\sqcup_a\underline\beta,\,\omega} &\hbox{ if }\zeta=\underline\alpha(a),\\
0 &\hbox{ otherwise}.
\end{matrix}
\right.
\end{equation}
subject to parallel and sequential associativity axioms, and there is a unit $e:\mathbf 1\to\mathcal P$. Informally, the partial composition $\circ_a$ is nontrivial if and only if the output color of the second term matches the input color of the first term corresponding to $a\in A$, otherwise $\circ_a$ takes values in the terminal object $0_{\mathcal C}$.\\

For any set map $\kappa:\Omega\to\Omega'$, the color change functor from $\Omega'$-colored species to $\Omega$-colored species is defined by
\begin{equation}\label{colorchange}
(\kappa^*\mathcal P)_{A,\underline\alpha,\omega}:=\mathcal P_{A,\,\kappa\circ\underline\alpha,\,\kappa(\omega)}
\end{equation}
for any $(A,\underline\alpha,\omega)\in\mathcal F_\Omega$. It respects both monoidal products $\boxtimes$, hence restricts from $\Omega'$-colored operads to $\Omega$-colored operads. In particular, the case when $\Omega'=\{*\}$ contains a unique element shows that any ordinary (monochromatic) operad $\mathcal Q$ can be promoted to an $\Omega$-colored operad $\mathcal Q^\Omega:=\kappa^*\mathcal Q$, with $\kappa:\Omega\to\{*\}$. The colored operad $\mathcal Q^\Omega$ is said to be \textsl{uniform}. This functor $\mathcal U:\mathcal Q\to\mathcal Q^\Omega$ is right-adjoint to the \textbf{completed forgetful functor} $\overline{\mathcal F}$ from $\Omega$-colored operads to ordinary operads, defined by
\begin{equation}\label{forgetcolors}
(\overline{\mathcal F}\mathcal P)_A:=\prod_{(\underline\alpha,\omega)\in\Omega^A\times \Omega}\mathcal P_{A,\underline\alpha,\omega}.
\end{equation}
\ignore{
Let $A$ be a vector space and let $\Omega$ be a set. We want $A\ot\bfk\Omega$ to be an algebra on some operad $\cal P$, together with a kind of $\Omega$-grading. Each $n$-ary operation $\ast$ gives rise to a family $(\ast_{\alpha_1,\ldots,\alpha_n})_{\alpha_j\in\Omega}$ of $n$-ary operations on $A$, and the grading yields $n$-ary operation $\circledast:\Omega^n\ra\Omega.$

Any relation between operations of the operad $\cal P$ yields correponding set-theoretical relations on $\Omega$. It makes $\Omega$ a (set-theoretical) algebra over some set-theoretical operad $ {\Large\textcircled{\small$\mathcal{P}$}}$ associated to $\cal P.$
}
\subsection{Categories of graded objects}\label{par:grad}
We keep the notations of the previous paragraph. The category $\mathcal C_\Omega$ of $\Omega$-graded objects \cite[Paragraph 2.2]{A2020} is the category of collections $(V_\omega)_{\omega\in\Omega}$ of objects of $\mathcal C$. A $\mathcal C_\Omega$-morphism
$$\varphi:(V_\omega)\longrightarrow (W_\omega)$$
is a collection $(\varphi_\omega)_{\omega\in\Omega}$ of $\mathcal C$-morphisms $\varphi_\omega:V_\omega\to W_\omega$.
This is not a monoidal category: indeed, the tensor product of two $\Omega$-graded objects is a collection indexed by $\Omega\times\Omega$.
\begin{remark}
In the case when $\Omega$ is a semigroup, categories of $\Omega$-graded objects can be given a monoidal structure by means of the Cauchy product \cite[Paragraph 2.2]{A2020}. We do not have this tool at our disposal here.
\end{remark}

A well-known example of $\Omega$-colored operad (in a bicomplete category $\mathcal C$ with internal Hom, i.e. such that $\hbox{Hom}(V,W)$ is an object of $\mathcal C$ for any pair $(V,W)$ of objects) is given by $\hbox{End}(\mathcal V)$ where $\mathcal V=(V_\omega)_{\omega\in\Omega}$ is an $\Omega$-graded object:
\begin{equation}\label{end-graded}
\hbox{End}(\mathcal V)_{A,\underline\alpha,\omega}:=\hbox{Hom}_{\mathcal C}\left(\bigotimes_{a\in A}V_{\underline\alpha(a)},\,V_\omega\right).
\end{equation}
Details are standard and left to the reader. An \textsl{algebra over an $\Omega$-colored operad $\mathcal P$} is an $\Omega$-graded object $\mathcal V$ together with a morphism of colored operads $\Phi:\mathcal P\to\hbox{End}(\mathcal V)$.
\begin{defn}\cite[Paragraph 2.2]{A2020}
An $\Omega$-graded object $\mathcal V=(V_\omega)_\omega$ is \textbf{uniform} if all homogeneous components are identical, i.e. if there is an object $V$ of $\mathcal C$ such that $V_\omega=V$ for any $\omega\in\Omega$. We write $\mathcal V=\mathcal U(V)$ in this case. This defines a functor $\mathcal U:\mathcal C\to\mathcal C_\Omega$, which has a right adjoint, the forgetful functor $\mathcal F:\mathcal C_\Omega\to\mathcal C$ defined by
$$\mathcal F(\mathcal V):=\coprod_{\omega\in\Omega}V_\omega,$$
which consists in forgetting the $\Omega$-grading \cite[Paragraph 2.4]{A2020}. It has also a left adjoint, the completed forgetful functor $\overline{\mathcal F}:\mathcal C_\Omega\to\mathcal C$ defined by
$$\overline{\mathcal F}(\mathcal V):=\prod_{\omega\in\Omega}V_\omega,$$
which consists taking the completion with respect to the $\Omega$-grading and then forgetting it.
\end{defn}

\section{Two-parameter $\Omega$-pre-Lie algebras}\label{sect:prelie}
The pre-Lie operad is no longer a set operad, hence new phenomena arise when seeking a compatible structure on the parameter set $\Omega$. Indeed, four different associated set operads are involved. The first one is the well-known associative operad. The second one is the operad governing rings with the twist-associativity condition $x(yz)=(yx)z$, also known as Thedy rings \cite{T1967}. The third one is the operad governing rings with both NAP relation $x(yz)=y(xz)$ and NAP' relation $(xy)z=(yx)z$. The fourth one is the operad governing rings with all previous relations at once, this is the well-known Perm operad \cite{C2002, CL2001}. After explaining this phenomenon in some detail, we give an explicit description of the twist-associative operad in terms of ordered pairs of distinct elements, and an explicit description of the NAPNAP' operad in terms of corollas, in the same spirit F. Chapoton and M. Livernet proved that the pre-Lie operad is given by labeled rooted trees \cite{CL2001}.
\subsection{Four possibilities}\label{par:four}
Let $A$ be a vector space and let $\Omega$ be a set with a binary operation $\br$. 
Suppose that $A\ot\bfk\Omega$ is endowed with an $\Omega$-graded pre-Lie product:
\begin{equation}
\mlabel{eq:pre} (x\ot\alpha)\rhd (y\ot\beta):=x\rhd_{\alpha,\beta}y\ot (\alpha \br\beta).
\end{equation}
The pre-Lie axiom
$$(x\ot\alpha)\rhd\Big((y\ot\beta)\rhd(z\ot\gamma)\Big)-\Big((x\ot\alpha)\rhd(y\ot\beta)
\Big)\ot(z\ot\gamma)=(y\ot\beta)\rhd\Big((x\ot\alpha)\rhd(z\ot\gamma)\Big)
-\Big((y\ot\beta)\rhd (x\ot\alpha)\Big)\rhd(z\ot\gamma)$$
together with the $\Omega$-grading are equivalent to
\begin{align}
&x\rhd_{\alpha,\beta\br\gamma}(y\rhd_{\beta,\gamma}z)
\ot\Big(\alpha\br(\beta\br\gamma)\Big)
-(x\rhd_{\alpha,\beta}y)
\rhd_{\alpha\br\beta,\gamma}z\ot\Big((\alpha\br\beta)\br\gamma\Big)\nonumber\\
=& \ y\rhd_{\beta,\alpha\br\gamma}(x\rhd_{\alpha,\gamma}z)
\ot\Big(\beta\br(\alpha\br\gamma)\Big)
-(y\rhd_{\beta,\alpha}x)\rhd_{\beta\br\alpha,\gamma}z
\ot\Big((\beta\br\alpha)\br\gamma\Big).\mlabel{eq:rhd}
\end{align}

\noindent Eq.~(\mref{eq:rhd}) induces four possible different cases.\\

\noindent {\bf Case 1:} let $\alpha\br(\beta\br \gamma)=(\alpha\br\beta)\br\gamma$ for $\alpha,\beta,\gamma\in\Omega.$ Thus $\Omega$ is a semigroup. Then
\begin{equation*}
x\rhd_{\alpha,\beta\br\gamma}(y\rhd_{\beta,\gamma}z)
=(x\rhd_{\alpha,\beta}y)\rhd_{\alpha\br\beta,\gamma}z,\,\text{ for }\, \alpha,\beta,\gamma\in\Omega
\end{equation*}
and we recover the notion of family associative algebra.\\

\noindent {\bf Case 2:} let
\begin{equation}\label{twist}
\alpha\br(\beta\br\gamma)=(\beta\br\alpha)\br\gamma
\end{equation}
for $\alpha,\beta,\gamma\in\Omega.$ Then $\Omega$ is a kind of ``twisted associative semigroup", a notion which has received little attention in the literature (see however \cite{T1967} and \cite{W2009}). We have then
\begin{equation}\label{family-twisted}
x\rhd_{\alpha,\beta\br\gamma}(y\rhd_{\beta,\gamma}z)
=-(y\rhd_{\beta,\alpha}x)\rhd_{\beta\br\alpha,\gamma}z
\end{equation}
and we recover a notion of ``family twisted associative algebra" modulo a minus sign.\\

We now give examples of twisted semigroups, i.e. sets endowed with a binary product $\br$ verifying Eq.~(\mref{twist}): for any set $D$ we consider the set $\mathbb N^D$ of maps form $D$ into the set $\mathbb N=\{0,1,2,3,\ldots\}$ of nonnegative integers. Such a map will be denoted by $\alpha=(\alpha_d)_{d\in D}$. For any $d\in D$, we denote by $\delta_d$ the map such that $\delta_d(d)=1$ and $\delta_d(e)=0$ for any $e\in D\setminus\{d\}$. Any element $\alpha\in\mathbb N^D$ will be represented by the monomial $X^\alpha$ in $d$ variables defined by
$$X^\alpha:=\prod_{d\in D}X_d^{\alpha_d}.$$
Now consider the set
$$T_D:=D\times D\times\mathbb N^D.$$
A generic element of $T_D$ will be denoted by $dd'X^\alpha$ with $d,d'\in D$ and $\alpha\in\mathbb N^D$. Let us now define the product by
$$dd'X^\alpha\br ee'X^\beta:=d'e'X^{\alpha+\delta_d+\beta+\delta_e}.$$
The product verifies Equation (\mref{twist}). Indeed, an easy computation yields
$$dd'X^\alpha\br (ee'X^\beta\br ff'X^\gamma)=(ee'X^\beta\br dd'X^\alpha)\br ff'X^\gamma=d'f'X^{\alpha+\beta+\gamma+\delta_d+\delta_e+\delta_{e'}+\delta_f}.$$
We prove in Paragraph \ref{par:twist} that $T_D$ is the free twisted semigroup generated by $D$, and we give an explicit description of the twist-associative set operad.\\

\noindent {\bf Case 3:} let
\begin{equation}
\left\{
\begin{array}{ll}
\alpha\br(\beta\br\gamma)=\beta\br(\alpha\br\gamma),\\
(\alpha\br\beta)\br\gamma=(\beta\br\alpha)\br\gamma.
\end{array}
\right .
\mlabel{eq:nap}
\end{equation}
The first relation is the NAP condition. We call NAP' the second condition, and we call $\Omega$ a NAPNAP' set. Then
 \begin{equation*}
\left\{
\begin{array}{ll}
x\rhd_{\alpha,\beta\br\gamma}(y\rhd_{\beta,\gamma}z)
=y\rhd_{\beta,\alpha\br\gamma}(x\rhd_{\alpha,\gamma}z),\\
(x\rhd_{\alpha,\beta}y)\rhd_{\alpha\br\beta,\gamma}z
=(y\rhd_{\beta,\alpha}x)\rhd_{\beta\br\alpha,\gamma}z.
\end{array}
\right .
\end{equation*}
We obtain what we shall call \emph{family NAPNAP'} algebras.

We give an example of set endowed with a binary product $\br$ verifying Eq.~(\mref{eq:nap}), namely the set $\underline{\mathbb N}$ of multisets of positive integers (including the empty multiset \o). Let $\underline n:=\{n_1,\ldots,n_k\}, \underline p=\{p_1,\ldots,p_\ell\}$ and $\underline q=\{q_1,\ldots,q_m\}$ be three elements of $\underline{\mathbb N}$. Define the product $\br$ by
\begin{equation*}
  \underline n\br \underline p:=\{n_1+\cdots+n_k+1,p_1,\ldots,p_\ell\}.
\end{equation*}
Then,
\begin{align*}
\underline n\br(\underline p\br \underline q)&=\underline p\br(\underline n\br \underline q)=\{n_1+\cdots +n_k+1,p_1+\cdots +P_\ell+1,q_1,\ldots,q_m\}\\
(\underline n\br \underline p)\br \underline q&=(\underline p\br \underline n)\br \underline q=\{n_1+\cdots +n_k+p_1+\cdots +p_\ell+2,q_1,\ldots,q_m\}.
\end{align*}
We'll prove in Paragraph \ref{par:corolla} that this is the free NAPNAP' set generated by the element \o, by giving an explicit description of the NAPNAP' operad. \\

\noindent\textbf{Case 4:} for any $\alpha,\beta,\gamma \in \Omega$,
\begin{align*}
\alpha\br(\beta\br\gamma)=(\alpha\br\beta)\br\gamma=\beta\br(\alpha\br\gamma)=(\beta\br\alpha)\br\gamma,
\end{align*}
i.e. $\Omega$ is a set-theoretical Perm algebra. Then for any $x,y,z\in A$,
\begin{align*}
x\rhd_{\alpha,\beta\br\gamma}(y\rhd_{\beta,\gamma}z)
-(x\rhd_{\alpha,\beta}y)
\rhd_{\alpha\br\beta,\gamma}z=& \ y\rhd_{\beta,\alpha\br\gamma}(x\rhd_{\alpha,\gamma}z)
-(y\rhd_{\beta,\alpha}x)\rhd_{\beta\br\alpha,\gamma}z.
\end{align*}
This relation is very similar to the pre-Lie one,  and deserves the name "pre-Lie family". We address this last case in Paragraph \ref{par:perm}

\subsection{The twist-associative operad $\mathbf{TAs}$}\label{par:twist}

\ignore{
\begin{defn}
 A species in the symmetric tensor category $\cal{C}$~\cite{Mac}  is a contravariant functor $E$ from the category of finite sets $E_{fin}$ with bijections to
$\cal{C}$. Thus, a species $E$ provides an object $E_A$ for any finite set $A$ and an isomorphism $E_\phi:E_B\ra E_A$ for any bijection $\phi:A\ra B.$
\mlabel{defn:spe}
\end{defn}
%
\begin{defn}
An operad $\cal{P}=(\cal{P}, \gamma,\eta)$ is a species $A\mapsto \cal P_A$ endowed with a monoid structure, i.e. an associative composition map $\gamma:\cal{P}\circ\cal{P}\ra\cal{P}$ and a unit map $\eta:{\bf I}\ra \cal{P}$ where ${\bf I}$ is the species defined by ${\bf I}_A=0$ for $|A|\neq 1$ and ${\bf I}_{\{*\}}=1_{\mathcal C}$.
\end{defn}
}
\begin{defn}
Let $\cal{P}$ be the set species of non-diagonal ordered pairs, defined by
\begin{eqnarray*}
\cal P_{\{*\}}&=&\mathbf 1,\\
\cal P_A&=&\{(a',a'')\in A\times A, \,a'\neq a''\}.
\end{eqnarray*}
for any finite set $A$ of cardinal $\ge 2$. For any bijection $\phi:A\ra B$ where $B$ is a finite set of the same cardinality than $A$, the relabeling isomorphism $\cal P_\phi:\cal P_B\ra \cal P_A$ is defined by
\begin{equation*}
\cal P_\phi(b',b'')=\Big(\phi^{-1}(b'),\phi^{-1}(b'')\Big).
\end{equation*}
\end{defn}

\begin{defn}
Let $A$ and $B$ be two finite sets.
Define partial compositions
\begin{equation*}
\circ_a: \calp_{A}\ot\calp_{B} \quad \ra\calp_{A\sqcup B\backslash\{a\}},\,\text{ for }\, a\in A.
\end{equation*}
as follows: for any ordered pair $(a',a'')\in \cal P_A$ and $(b',b'')\in \cal P_B$, we set
\begin{equation}
\mlabel{eq:com}(a',a'')\circ_a(b',b'')=
\left\{
\begin{array}{ll}
(b'',a''), & \text{ if } a=a';\\
(a',b''), & \text{ if } a=a'';\\
(a',a''), & \text{ if } a\notin \{a',a''\}.
\end{array}
\right .
\end{equation}
Partial compositions are extended to singletons by setting $\mathbf 1$ as the unit.
\end{defn}

\begin{prop}\mlabel{prop:operad}
The species $\cal P$ together with the partial compositions $\circ_a$ defined by Eq.~(\mref{eq:com}) is an operad.
\end{prop}

\begin{proof}
Let $A, B, C$ be three sets of cardinal $\ge 2$, and let $x=(a',a'')\in A\times A, y=(b',b'')\in B\times B, z=(c',c'')\in C\times C.$ When we prove sequential associativity, \ignore{and parallel associativity} there are nine cases to consider.
$$\begin{tabular}{|c|c|c|c|}
  \hline
  1 & $a=a',b=b'$ & 2 & $a=a',b=b''$\\
  \hline
  3 & $a=a',b\notin\{b',b''\}$ & 4 & $a=a'',b=b'$ \\
  \hline
  5 & $a=a'',b=b''$ & 6 & $a=a'',b\notin\{b',b''\}$ \\
  \hline
  7 & $a\notin \{a',a''\},b=b'$ & 8 & $a\notin \{a',a''\},b=b''$ \\
  \hline
  9 & $a\notin \{a',a''\},b\notin\{b',b''\}$ &  &  \\
  \hline
\end{tabular}$$
\hspace{4cm} \title{ Table: the nine cases for sequential associativity}\mlabel{title:table}\\

\noindent The case-by-case proof is displayed on the following table:
$$\begin{tabular}{|c|c|c|}
  \hline
  & $(x\circ_a y)\circ_b z=\Big((a',a'')\circ_a(b',b'')\Big)\circ_b(c',c'')$ & $x\circ_a (y\circ_b z)=(a',a'')\circ_a\Big((b',b'')\circ_b(c',c'')\Big)$ \\
  \hline
  \hline
  1& $(b'',a'')\circ_b(c',c'')=(b'',a'')$ & $(a',a'')\circ_a(c'',b'')=(b'',a'')$ \\
  \hline
  2& $(b'',a'')\circ_b(c',c'')=(c'',a'')$ & $(a',a'')\circ_a(b',c'')=(c'',a'')$ \\
  \hline
  3& $(b'',a'')\circ_b(c',c'')=(b'',a'')$ & $(a',a'')\circ_a(b',b'')=(b'',a'')$ \\
  \hline
  4& $(a',b'')\circ_b(c',c'')=(a',b'')$ & $(a',a'')\circ_a(c'',b'')=(a',b'')$ \\
  \hline
  5& $(a',b'')\circ_b(c',c'')=(a',c'')$ & $(a',a'')\circ_a(b',c'')=(a',c'')$ \\
  \hline
  6& $(a',b'')\circ_b(c',c'')=(a',b'')$ & $(a',a'')\circ_a(b',b'')=(a',b'')$ \\
  \hline
  7& $(a',a'')\circ_b(c',c'')=(a',a'')$ & $(a',a'')\circ_a(c'',b'')=(a',a'')$ \\
  \hline
  8& $(a',a'')\circ_b(c',c'')=(a',a'')$ & $(a',a'')\circ_a(b',c'')=(a',a'')$ \\
  \hline
  9& $(a',a'')\circ_b(c',c'')=(a',a'')$ & $(a',a'')\circ_a(b',b'')=(a',a'')$ \\
  \hline
\end{tabular}$$

Let us now turn to parallel associativity. There are seven cases to consider. Here $a$ and $\overline a$ stand for two different elements in $A$.
$$\begin{tabular}{|c|c|c|c|}
  \hline
  1 & $a=a',\ \overline a=a''$ & 2 & $a=a',\ \overline a\notin\{a',a''\}$\\
  \hline
  3 & $a=a'',\ \overline a=a'$ & 4 & $a=a'',\ \overline a\notin\{a',a''\}$ \\
  \hline
  5 & $a\notin\{a',a''\},\ \overline a=a'$ & 6 & $a\notin\{a',a''\},\ \overline a=a''$ \\
  \hline
  7 & $a\notin \{a',a''\},\ \overline a\notin\{a',a''\}$ &   &  \\
  \hline
\end{tabular}$$
\hspace{4cm} \title{ Table: the seven cases for parallel associativity}\mlabel{title:tablebis}\\

\noindent The case-by-case proof is displayed on the following table:
$$\begin{tabular}{|c|c|c|}
  \hline
  & $(x\circ_a y)\circ_{\overline a} z=\Big((a',a'')\circ_a(b',b'')\Big)\circ_{\overline a}(c',c'')$ & $(x\circ_{\overline a} z) \circ_a y=\Big((a',a'')\circ_{\overline a}(c',c'')\Big)\circ_{a}(c',c'')$ \\
  \hline
  \hline
  1& $(b'',a'')\circ_{\overline a}(c',c'')=(b'',c'')$ & $(a',c'')\circ_a(b',b'')=(b'',c'')$ \\
  \hline
  2& $(b'',a'')\circ_{\overline a}(c',c'')=(b'',a'')$ & $(a',a'')\circ_a(b',b'')=(b'',a'')$ \\
  \hline
  3& $(a',b'')\circ_{\overline a}(c',c'')=(c'',b'')$ & $(c'',a'')\circ_a(b',b'')=(c'',b'')$ \\
  \hline
  4& $(a',b'')\circ_{\overline a}(c',c'')=(a',b'')$ & $(a',a'')\circ_a(b',b'')=(a',b'')$ \\
  \hline
  5& $(a',a'')\circ_{\overline a}(c',c'')=(c'',a'')$ & $(c'',a'')\circ_a(b',b'')=(c'',a'')$ \\
  \hline
  6& $(a',a'')\circ_{\overline a}(c',c'')=(a',c'')$ & $(a',c'')\circ_a(b',b'')=(a',c'')$ \\
  \hline
  7& $(a',a'')\circ_{\overline a}(c',c'')=(a',a'')$ & $(a',a'')\circ_a(b'',b'')=(a',a'')$ \\
  \hline
\end{tabular}$$
Extending to the case where $A$, $B$ or $C$ has only one element is straightforward and left to the reader.
\end{proof}
\begin{prop}\mlabel{prop:twist}
Let $A=\{1,2\}$, let $\mu=(1,2)\in\cal P_A$, and let $\overline \mu=(2,1)$ be the other element of $\cal P_A$ obtained by permutation. The twist-associativity relation
\begin{equation}\label{relations-operad}
r=\mu\circ_2\mu-\mu\circ_1\bar{\mu}=0
\end{equation}
holds in the operad $\cal P$.
\end{prop}

\begin{proof}
Denoting by $\{a,b\}$ another copy of $A$ (identifiying $a$ with $1$ and $b$ with $2$), both three-element sets $A\sqcup_2 A$ and $A\sqcup_1 A$ must be identified by means of the bijection
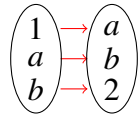
\usetikzlibrary{fit,shapes.geometric}
\[\tikz[x=0.5cm,y=0.4cm,every node/.style={minimum width=1.5em,minimum height=2ex,inner sep=0pt}]{
\node (l1) at (0,0) {$1$};
\node (l2) at (0,-1) {$a$};
\node (l3) at (0,-2) {$b$};
\node (r1) at (2,0) {$a$};
\node (r2) at (2,-1) {$b$};
\node (r3) at (2,-2) {$2$};
\draw[red,->] (l1)--(r1);
\draw[red,->] (l2)--(r2);
\draw[red,->] (l3)--(r3);
\node[ellipse,draw,fit=(l1)(l2)(l3),inner sep=-2.5pt]{};
\node[ellipse,draw,fit=(r1)(r2)(r3),inner sep=-2.5pt]{};
}
\]
in order to make Equation \eqref{relations-operad} consistent. We get then
\begin{align*}
\mu\circ_2\mu&=(1,2)\circ_2(a,b)=(1,b)\in\calp_{\{1,a,b\}},\\
\mu\circ_1\bar{\mu}&=(1,2)\circ_1(b,a)=(a,2)\in\calp_{\{a,b,2\}},
\end{align*}
hence $\mu\circ_2\mu=\mu\circ_1\bar{\mu}$ modulo the identification above.
\end{proof}
\noindent For later use, for any $A,B$ finite sets we define the product $\br:\calp_A\otimes\calp_B\to\calp_{A\sqcup B}$ by
$$\alpha\br\beta:=(\mu\circ_1\alpha)\circ_2\beta.$$
An easy computation yields:
\begin{eqnarray}
\mathbf 1\br\mathbf 1&=&\mu,\\
\mathbf 1\br(x,y)&=&(*,y),\label{starleft}\\
(x,y)\br\mathbf 1&=&(y,*),\\
(x,y)\br(z,t)&=&(y,t)
\end{eqnarray}
for any $x,y\in A$ and $z,t\in B$ with $x\neq y$ and $z\neq t$. It is easily checked that the product verifies the twist-associative identity
\begin{equation}\label{tai}
\alpha\br(\beta\br \gamma)=(\beta\br \alpha)\br \gamma
\end{equation}
for any finite sets $A,B,C$ and for any $\alpha\in\calp_A$, $\beta\in\calp_B$ and $\gamma\in\calp_C$.
\begin{theorem}\label{thm:twist}
The operad $\cal P$ of non-diagonal ordered pairs is isomorphic to the twist-associative operad $\cal T:=\mathcal M\Big/\langle r\rangle$.
\end{theorem}

\begin{proof}
We still adopt the notations in the proof of Proposition~\mref{prop:twist}.
The twist-associative operad is defined as the quotient of the magmatic operad $\mathcal M$ (the free operad generated by a single binary operation $\nu$) by the ideal $\langle r\rangle$ generated by the twist-associative relation $r=\nu\circ_2\nu-\nu\circ_1\bar{\nu}$. Let $A,B,C$ be three finite sets. Defining $\widetilde{\nu}$ as the image of $\nu$ in the quotient, we have
\begin{equation}\label{taibis}
\alpha\rhd(\beta\rhd\gamma)=(\beta\rhd\alpha)\rhd\gamma
\end{equation}
for any $\alpha\in\mathcal T_A$, $\beta\in\mathcal T_B$ and $\gamma\in\mathcal T_C$, where $\rhd$ is defined by
$$\alpha\rhd\beta:=(\widetilde \nu\circ_1\alpha)\circ_2\beta.$$
As the ordered pair $\mu=(1,2)$ verifies the twist-associative relation \eqref{relations-operad}, there is a unique surjective operad morphism $\Phi:\cal T\ra\cal P$ such that $\Phi(\widetilde\nu)=\mu$. It obviously verifies
\begin{equation}
\Phi(\alpha\rhd\beta)=\Phi(\alpha)\br\Phi(\beta)
\end{equation}
for any $\alpha\in\cal T_A$ and $\beta\in\cal T_B$. Let us prove that $\Phi$ is bijective. Define $\Psi_A:\calp_A\ra \cal T_A$ by induction on the arity $n=|A|\geq 2$. For $n=1$ we set $\Psi(\mathbf 1)=\mathbf 1$, and for $n=2$ it amounts to $\Psi(\mu)=\widetilde\nu$. Suppose that the inverse $\Psi$ of $\Phi$ is well-defined (and hence bijective) up to arity $n$, and let $A$ be of cardinality $n+1$. From \eqref{starleft}, any $(x,y)\in\calp_A$ can be written $(x,y)=\mathbf 1\br(x',y)$, where $\mathbf 1\in\calp_{\{x\}}$ and $(x',y)\in\calp_{A\setminus\{x\}}$. Hence we necessarily have
$$\Psi(x,y)=\mathbf 1\rhd\Psi(x',y).$$
It is well defined because it does not depend on the choice of $x'$. Indeed, if another choice $x''$ is possible, then
$$(x,y)=\mathbf 1\br(x'',y)=\mathbf 1\br\big(\mathbf 1\br (x',y)\big),$$
hence
\begin{eqnarray*}
\mathbf 1\rhd\Psi(x'',y)&=&\mathbf 1\rhd\big(\mathbf 1\rhd \Psi(x',y)\big),\hskip 6mm (x',y)\in A\setminus\{x,x''\}\\
&=&\mathbf 1\rhd\big(\mathbf 1\rhd \Psi(x'',y)\big),\hskip 6mm (x'',y)\in A\setminus\{x,x'\}\hbox{ (by induction hypothesis)},\\
&=&\mathbf 1\rhd\Psi(x',y)\hskip 6mm\hbox{(again by induction hypothesis)}.
\end{eqnarray*}
We have
$$\Phi\Psi(x,y)=\mathbf 1\rhd \Phi\Psi(x',y)=\mathbf 1\rhd (x',y)=(x,y)$$
by induction hypothesis, hence $\Phi_A\Psi_A=\hbox{Id}_{\calp_A}$. Furthermore, for any partition $A=B\sqcup C$ and for any $\beta\in\cal P_B, \gamma\in\cal P_C$ we have
$$\Psi(\beta\br\gamma)=\Psi(\beta)\rhd\psi(\gamma).$$
This is easily proven by induction on the cardinality of $B$, the case $|B|=1$ being equivalent to the definition of $\Psi$: if $|B|\ge 2$ we write $\beta=\mathbf 1\br \beta'$ and then
\begin{eqnarray*}
\Psi(\beta\br\gamma)&=&\Psi\big((\mathbf 1\br \beta')\br\gamma\big)\\
&=&\Psi\big(\beta'\br(\mathbf 1\br\gamma)\big) \hbox{ (by \eqref{tai})}\\
&=&\Psi(\beta')\rhd\Psi(\mathbf 1\br\gamma) \hbox{ (by induction on }|B|)\\
&=&\Psi(\beta')\rhd\big(\mathbf 1\rhd\Psi(\gamma)\big)\\
&=&\big(\mathbf 1\rhd\Psi(\beta')\big)\rhd\Psi(\gamma)\hbox{ (by \eqref{taibis})}\\
&=&\Psi(\beta)\rhd\Psi(\gamma).
\end{eqnarray*}

Now any $\alpha\in\cal T_A$ can be written $\beta\rhd\gamma$ with $\beta\in\cal T_B$, $\gamma\in\cal T_C$, where $B$ and $C$ are two finite sets of cardinality $\le n$ such that $A=B\sqcup C$. We have then
$$\Psi\Phi(\alpha)=\Psi\Phi(\beta\rhd\gamma)=\Psi\Phi(\beta)\rhd\Psi\Phi(\gamma)=\beta\rhd\gamma=\alpha$$
again by induction hypothesis, hence $\Psi_A\Phi_A=\hbox{Id}_{\cal T_A}$. This ends up the proof of Theorem \ref{thm:twist}.
\end{proof}
Let us remark that, forgetting the labels and putting instead a decoration by a given set $D$, we recover the description of the free twisted associative semigroup generated by $D$ given in Paragraph \ref{par:four}.
\subsection{The operad $\mathbf{NAPNAP}'$ of corollas}\label{par:corolla}
\begin{defn}
A \textbf{corolla structure} $\beta$ on a finite set $B$ is a quasi-order admitting one unique minimum $r$, such that any element different from $r$ is a maximum.
\end{defn}
The unique minimum $r$ is the \textsl{root} of the corolla. Any $b\neq r$ verifies $r\le b$ but never $b\le r$. The non-root elements are partitioned into \textsl{branches} $B_1,\ldots,B_p$, which are the equivalence classes (excluding the one of the root) under the relation $\sim$ defined by $b\sim b'$ if and only if $b\le b'$ and $b'\le b$. We shall write
$$
\beta=[B_1,\ldots,B_p]_r.
$$
For example, on the finite set $B:\{a,b,c,d,e,f,g\}$, the notation $\beta=[\{b,c\},\{d\},\{e,f,g\}]_a$ stands for the corolla
$$
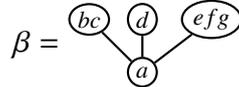
\beta=\treeo { \eoo `a@
\draw (o) \eee -1,1`l`bc@ (o)--(l) ;
\draw (o) \eee 0,1`a`d@ (o)--(a) ;
\draw (o) \eee 1.3,1`r`{efg}@ (o)--(r) ;  }
$$
Let $\mathbb K_B$ be the set of corolla structures on $B$. This forms a set species: any bijection $\varphi:B\to C$ induces a bijection $\mathbb K_\varphi:\mathbb K_C\to\mathbb K_B$ by relabeling.\\

Now let us define the operad structure. Let $B,C$ be two finite sets, let $b\in B$, let $\beta\in\mathbb K_B$ and $\gamma\in\mathbb K_C$. Let $r$ be the root of the corolla $\gamma$. The partial composition $\beta\circ_b\gamma:\mathbb K_B\times\mathbb K_C\to\mathbb K_{B\sqcup_b C}$ is defined as follows:
\begin{itemize}
\item if $b$ is the root of $\beta$, then $\beta\circ_b\gamma$ is the corolla on $B\sqcup_b C$ obtained by choosing $r$ as the root, and by keeping all branches in $B\sqcup_b C\setminus\{r\}$. In particular, elements in $B\setminus\{b\}$ and elements in $C\setminus\{r\}$ belong to different branches, and thus are uncomparable.
\item if $b$ is not the root of $\beta$, then $\beta\circ_b\gamma$ is the corolla on $B\sqcup_b C$ obtained by replacing $b$ by the whole $C$ in the branch of $b$.
\end{itemize}
Let us give an example for better understanding.
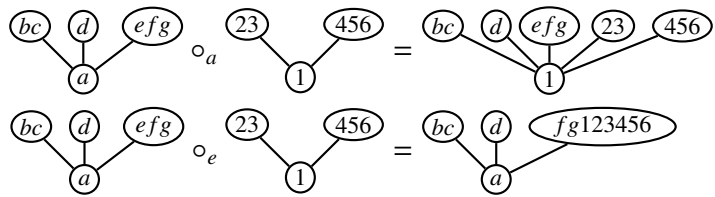
\begin{align*}
\treeo { \eoo `a@
\draw (o) \eee -1,1`l`bc@ (o)--(l) ;
\draw (o) \eee 0,1`a`d@ (o)--(a) ;
\draw (o) \eee 1.3,1`r`{efg}@ (o)--(r) ;  }
\circ_a \treeo{
\eoo`1@
\draw (o)
\eee-1,1`l`23@
(o)--(l)
;
\draw (o)
\eee1,1`r`456@
(o)--(r)
;
}
&=
\treeo { \eoo `1@
\draw (o) \eee -2,1`l1`bc@ (o)--(l1) ;
\draw (o) \eee -1,1`l2`d@ (o)--(l2) ;
\draw (o) \eee 0,1`a`efg@ (o)--(a) ;
\draw (o) \eee 2.5,1`r1`456@ (o)--(r1) ;
\draw (o) \eee 1.2,1`r2`23@ (o)--(r2) ;  }\\
\treeo { \eoo `a@
\draw (o) \eee -1,1`l`bc@ (o)--(l) ;
\draw (o) \eee 0,1`a`d@ (o)--(a) ;
\draw (o) \eee 1.3,1`r`{efg}@ (o)--(r) ;  }
\circ_e \treeo{
\eoo`1@
\draw (o)
\eee-1,1`l`23@
(o)--(l)
;
\draw (o)
\eee1,1`r`456@
(o)--(r)
;
}
&=\treeo { \eoo `a@
\draw (o) \eee -1,1`l`bc@ (o)--(l) ;
\draw (o) \eee 0,1`a`d@ (o)--(a) ;
\draw (o) \eee 2.0,1`r`{fg123456}@ (o)--(r) ;  }
\end{align*}
We leave it to the reader to show that $\mathbb K$ endowed with the partial compositions defined above is an operad, i.e. prove both sequential and parallel associativity axioms. Now define the product $\br$ on $\mathbb K$ by
$$\beta\br\gamma:=(\treeo{
\eoo`2@
\draw (o)
\eee0,1`l`1@
(o)--(l);} \circ_1\beta)\circ_2\gamma.$$
\begin{prop}
The product $\br$ verifies for any $\alpha,\beta,\gamma\in\mathbb K$:
\begin{enumerate}
\item $\alpha\br(\beta\br\gamma)=\beta\br(\alpha\br\gamma)$,
\item $(\alpha\br\beta)\br\gamma=(\beta\br\alpha)\br\gamma$.
\end{enumerate}
\end{prop}
\begin{proof}
Both sides of Equation (a) are equal to $\left(\Big(\treeo{
\eoo`3@
\draw (o)
\eee-1,1`l`1@
(o)--(l)
;
\draw (o)
\eee1,1`r`2@
(o)--(r)
;
}\circ_1\alpha\Big)\circ_2\beta\right)\circ_3\gamma$, and both sides of Equation (b) are equal to $\left(\Big(\treeo{
\eoo`3@
\draw (o)
\eee0,1`l`12@
(o)--(l);} \circ_1\alpha\Big)\circ_2\beta\right)\circ_3\gamma$. Details are left to the reader.
\end{proof}
\begin{theorem}\label{thm:corollas}
The operad $\mathbb K$ of corollas is the NAPNAP' operad.
\end{theorem}
\begin{proof}
The NAPNAP' operad is defined as the quotient of the magmatic operad $\mathcal M$ by the NAP and NAP' relations, namely
$$
\hbox{NAPNAP'}:=\mathcal M\Big /\Big\langle\mu\circ_2\mu-\tau_{12}(\mu\circ_2\mu),\,\, \mu\circ_1\mu-\tau_{12}(\mu\circ_1\mu)\Big\rangle.
$$
Defining $\overline\mu$ as the image of $\mu$ in the quotient, we further introduce the product $\rhd$ on the NAPNAP' operad itself, defined by
$$\alpha\rhd\beta:=(\overline\mu\circ_1\alpha)\circ_2\beta.$$
The NAP and NAP' relations for $\overline\mu$ yield analogous relations for $\rhd$, namely
\begin{enumerate}
\item $\alpha\rhd(\beta\rhd\gamma)=\beta\rhd(\alpha\rhd\gamma)$,
\item $(\alpha\rhd\beta)\rhd\gamma=(\beta\rhd\alpha)\rhd\gamma$.
\end{enumerate}
The corolla $\treeo{
\eoo`2@
\draw (o)
\eee0,1`l`1@
(o)--(l);}$ respects both NAP and NAP' relations, namely
$$
\treeo{
\eoo`r@
\draw (o)
\eee0,1`l`1@
(o)--(l);}\circ_r \treeo{
\eoo`s@
\draw (o)
\eee0,1`l`2@
(o)--(l);}=\tau_{12}\Big(\,\treeo{
\eoo`r@
\draw (o)
\eee0,1`l`1@
(o)--(l);}\circ_r \treeo{
\eoo`s@
\draw (o)
\eee0,1`l`2@
(o)--(l);}\,\Big)
=\treeo{
\eoo`s@
\draw (o)
\eee-1,1`l`1@
(o)--(l)
;
\draw (o)
\eee1,1`r`2@
(o)--(r)
;
}.
$$
and
$$
\treeo{
\eoo`r@
\draw (o)
\eee0,1`l`a@
(o)--(l);}\circ_a \treeo{
\eoo`2@
\draw (o)
\eee0,1`l`1@
(o)--(l);}=\tau_{12}\Big(\,\treeo{
\eoo`r@
\draw (o)
\eee0,1`l`a@
(o)--(l);}\circ_a \treeo{
\eoo`2@
\draw (o)
\eee0,1`l`1@
(o)--(l);}\,\Big)
=\treeo{
\eoo`r@
\draw (o)
\eee0,1`l`12@
(o)--(l);}.
$$
Hence the operad morphism $\widetilde\Phi$ from $\mathcal M$ onto $\mathbb K$ uniquely defined by $\widetilde\Phi(\mu)=\treeo{
\eoo`2@
\draw (o)
\eee0,1`l`1@
(o)--(l);}$ vanishes on the ideal generated by the NAP and NAP' relations, giving rise to the unique surjective operad morphism
$$\Phi:\hbox{NAPNAP'}\longrightarrow\mathbb K$$
such that $\Phi(\mu)=\treeo{
\eoo`2@
\draw (o)
\eee0,1`l`1@
(o)--(l);}$. It is obvious that $\Phi$ changes product $\rhd$ into product $\br$. It remains to prove that $\Phi$ is an isomorphism. We will prove the existence of an inverse $\Psi:\mathbb K_B\to\hbox{NAPNAP'}_B$ of $\Phi:\hbox{NAPNAP'}_B\to \mathbb K_B$ for any finite set $B$ by induction on the cardinal of $B$. The cases where $B$ has one or two elements are trivial. Suppose the result to be true up to $n$ elements, and let $B$ be of cardinal $n+1$. For any corolla structure $\beta$ on $B$, there is $r\in B$ and a partition $B\setminus\{r\}=B_1\sqcup\cdots\sqcup B_p$ such that
$$\beta=[B_1,\ldots,B_p]_r.$$
We now proceed by a secondary induction on $p$. If $p=1$, we have $\beta=[B_1]_r=\beta'\br\{r\}$, where $\beta'$ is any corolla structure on $B\setminus\{r\}$, and where we identify the one-element set $\{r\}$ with the only corolla structure which exists on it. We set:
$$\Psi(\beta')\rhd\mathbf 1,$$
$\Phi\Psi(\beta)=\Phi\Psi(\beta')\br\Phi\Psi(\{r\})=\beta$ by induction hypothesis. For $p\ge 2$ we have
$$
\beta=[B_1,\ldots,B_p]_r=\beta_1\br [B_2,\ldots,B_p]_r,
$$
where $\beta_1$ is any corolla structure on $B_1$. We can define by induction hypothesis:
$$\Psi(\beta):=\Psi(\beta_1)\rhd\Psi([B_2,\ldots,B_p]_r).$$
We have again $\Phi\Psi(\beta)=\beta$ for the same reasons. To make sure that $\Psi(\beta)$ is well-defined, one has to prove that the result is invariant under permutation of the $p$ branches. Invariance under permutation of the $p-1$ last ones is obvious by secondary induction hypothesis. To get invariance under permutation of $B_1$ and $B_2$, define
$$\Psi'(\beta):=\Psi(\beta_2)\rhd\Psi([B_1,B_3\ldots,B_p]_r)$$
where $\beta_2$ is any corolla structure on $B_2$. We have then
\begin{eqnarray*}
\Psi(\beta)&=&\Psi(\beta_1)\rhd\Psi([B_2,\ldots,B_p]_r)\\
&=&\Psi(\beta_1)\rhd\Big(\Psi(\beta_2)\rhd\Psi([B_3,\ldots,B_p]_r)\Big)\\
&=&\Big(\Psi(\beta_1)\rhd\Psi(\beta_2)\Big)\rhd\Psi([B_3,\ldots,B_p]_r)\\
&=&\Big(\Psi(\beta_2)\rhd\Psi(\beta_1)\Big)\rhd\Psi([B_3,\ldots,B_p]_r)\\
&=&\Psi(\beta_2)\rhd\Big(\Psi(\beta_1)\rhd\Psi([B_3,\ldots,B_p]_r)\Big)\\
&=&\Psi(\beta_2)\rhd\Psi([B_1,B_3\ldots,B_p]_r)\\
&=&\Psi'(\beta).
\end{eqnarray*}
Finally we also have $\Psi\Phi=\hbox{Id}_{\hbox{\tiny NAPNAP'}}$. It is easily proven by induction on arity, using $\Psi(\alpha\br\beta)=\Psi(\alpha)\rhd\Psi(\beta)$. This ends up the proof of Theorem \ref{thm:corollas}.
\end{proof}
\subsection{Two-parameter $\Omega$-pre-Lie algebras and the $\mathbf{Perm}$ operad}\label{par:perm}
Now we give the definition of two-parameter $\Omega$-pre-Lie algebras. This requires that the product $\br$ on $\Omega$ fulfils the requirements of the fourth case of Paragraph \ref{par:four}:
\begin{defn}
Let $\Omega$ be a \textsl{set-theoretical perm algebra} \mcite{CL2001, C2002}, i.e. a set with a product $\br$ such that
\begin{equation}\label{assoc-nap}
\alpha\br(\beta\br\gamma)=(\alpha\br\beta)\br\gamma=\beta\br(\alpha\br\gamma)=(\beta\br\alpha)\br\gamma
\end{equation}
for any $\alpha,\beta,\gamma\in\Omega$, i.e. we ask that the product $\br$ is both associative and NAP. A {\bf two-parameter $\Omega$-pre-Lie algebra} is a family $(A,(\rhd_{\alpha,\beta})_{\alpha,\beta\in\Omega})$ where $A$ is a vector space and $\rhd_{\alpha,\beta}:A\ot A\ra A,$ such that for any $x,y,z\in A$ and $\alpha,\beta\in\Omega,$ satisfying
  \begin{equation}\label{plf}
  x\rhd_{\alpha,\beta\br\gamma}(y\rhd_{\beta,\gamma}z)
-(x\rhd_{\alpha,\beta}y)
\rhd_{\alpha\br\beta,\gamma}z
=y\rhd_{\beta,\alpha\br\gamma}(x\rhd_{\alpha,\gamma}z)
-(y\rhd_{\beta,\alpha}x)\rhd_{\beta\br\alpha,\gamma}z.
  \end{equation}
\end{defn}
The Perm operad governing relations \eqref{assoc-nap} has been described in \cite{C2002}. In the species formalism, $\hbox{Perm}_A:=A$ for any finite set $A$, and the partial compositions are defined as follows: for any finite sets $A,B$, for any $a,a'\in A$ and $b'\in B$,
$$a'\circ_a b'=
\left\{
\begin{array}{ll}
b' & \text{ if } a=a';\\
a' & \text{ if } a\neq a'.
\end{array}
\right.
$$

Note that if $(\Omega,\br)$ is a set-theoretical Perm algebra, then it is also a semigroup, a twisted associative semigroup,
and a set-theoretical NAPNAP' algebra. Hence, we can consider the operad $\mathbf{As}_\Omega$ of family associative algebras
on $\Omega$ as in Case 1, $\mathbf{TAs}_\Omega$ of family twisted associative algebras on $\Omega$ defined by \eqref{family-twisted} as in Case 2,
$\mathbf{NAPNAP}'_\Omega$ of family NAPNAP' algebras on $\Omega$ as in Case 3, and 
$\mathbf{PreLie}_\Omega$ of family pre-Lie algebras on $\Omega$ as in Case 4. Then, in an immediate way:

\begin{prop}
For any set-theoretical Perm algebra $\Omega$, we have the two following diagrams.
 \diagramme
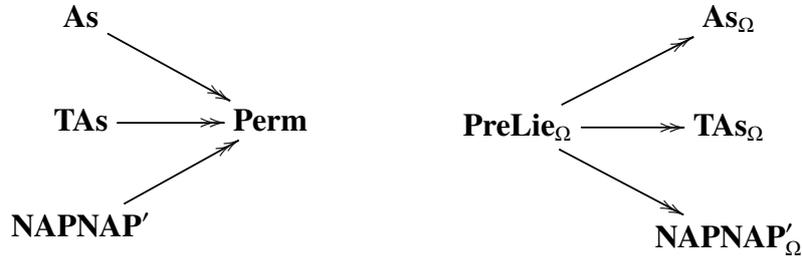
{
 \xymatrix{
 \mathbf{As}\ar@{->>}[dr] &\\
\mathbf{TAs}\ar@{->>}[r] &\mathbf{Perm}\\
\mathbf{NAPNAP}'\ar@{->>}[ur]&
 }
 \hskip 18mm
  \xymatrix{
& \mathbf{As}_\Omega\\
\mathbf{PreLie}_\Omega\ar@{->>}[r]\ar@{->>}[ur]\ar@{->>}[dr] &\mathbf{TAs}_\Omega\\
 &\mathbf{NAPNAP}'_\Omega
 }
 }
\noindent The four operads in the first diagram are set operads, and all arrows are surjective.
\ignore{There exists three  surjective maps of operads
\begin{align*}
&\left\{\begin{array}{rcl}
\mathbf{PreLie}_\Omega&\rightarrow&\mathbf{As}_\Omega\\
\rhd_{\alpha,\beta}&\mapsto&\rhd_{\alpha,\beta},
\end{array}\right.&
&\left\{\begin{array}{rcl}
\mathbf{PreLie}_\Omega&\rightarrow&\mathbf{TAs}_\Omega\\
\rhd_{\alpha,\beta}&\mapsto&\rhd_{\alpha,\beta},
\end{array}\right.&
&\left\{\begin{array}{rcl}
\mathbf{PreLie}_\Omega&\rightarrow&\mathbf{NAPNAP'}_\Omega\\
\rhd_{\alpha,\beta}&\mapsto&\rhd_{\alpha,\beta}.
\end{array}\right.
\end{align*}
}
\end{prop}

\section{Color-mixing operads and family algebraic structures}\label{sect:multigraded}
We follow the lines of \cite[Section 2]{A2020}, except that we consider gradings taking values in an arbitrary set $\Omega$ rather than in a semigroup. The key point is that if an algebraic structure on a graded object is compatible with the grading in a natural sense, this algebraic structure in turn provides an algebraic structure on $\Omega$. For example, a degree-compatible associative algebra structure on a graded object yields a semigroup structure on $\Omega$, a degree-compatible dendriform (resp. duplicial) algebra structure on a graded object yields a diassociative (resp. duplicial) semigroup structure on $\Omega$, and so on.

\subsection{Color-mixing operads: the principle}
Let $\Omega$ be a set of colors and $\mathcal C$ be a bicomplete monoidal category. Keeping the notations of Paragraphs \ref{par:colored} and \ref{par:grad}, for any operad $\mathcal P$, all components $\mathcal P^\Omega_{A,\underline\alpha,\omega}$ of the colored operad $\mathcal P^\Omega$ are isomorphic once the finite set $A$ of inputs is fixed. This reflects the fact that, if a $\mathcal P$-algebra $V$ is a coproduct
$$V=\coprod_{\omega\in\Omega}V_\omega,$$
any operation $\mu:V^{\otimes A}\to V$ with $A$-indexed inputs has a priori nonzero components
$$\mu_{\underline\alpha,\omega}:\bigotimes_{a\in A} V_{\underline\alpha(a)}\longrightarrow V_\omega$$
for any $\underline\alpha\in\Omega^A$ and $\omega\in\Omega$. This clearly contradicts the principle outlined in the introducting paragraph of this section, according to which the graded object $\mathcal V=(V_\omega)_{\omega\in\Omega}$ should be not only a $\mathcal P^\Omega$-algebra, but also a "graded $\mathcal P$-algebra" in some sense. This means that the output color $\omega$ should be a combination of the input colors $\underline\alpha$ in a way prescribed by the operad $\mathcal P$.
\subsection{Color-mixing linear operads}\label{par:cmlin}
From now on, we stick to the case when $\mathcal C$ is the category of vector spaces over some field $\bfk$. The coproduct is now given by the usual direct sum $\oplus$. Guided by the dendriform and the pre-Lie examples detailed in the previous sections, we see that the color set $\Omega$ will be endowed with a  \cP -algebra structure, where \cP is a set operad derived from the linear operad $\mathcal P$. For example, if $\mathcal P$ is the dendriform operad, \cP is the the diassociative operad, and if $\mathcal P$ is the pre-Lie operad, \cP is the Perm operad. When $\mathcal P$ is the linearization of a set operad $\mathbf P$, we should get \cP$=\mathbf P$, as the duplicial example suggests.\\

\noindent We suppose that $\mathcal P$ is of finite presentation, i.e. it can be written as
\begin{equation}
\mathcal P=\mathcal M_E\Big /\mathcal R,
\end{equation}
where $E$ is a set species of generators, $\mathbf M_E$ is the free set operad generated by $E$, and $\mathcal R$ is the operadic ideal of the linear operad $\mathcal M_E=\bfk.\mathbf M_E$ generated by a finite linearly independent collection $\mu^1,\ldots,\mu^N$ of elements. Each of these elements can be written as
$$\mu^i=\sum_{j=1}^{k_i}\lambda_j^i\mu_j^i,$$
where $(\mu_j^i)_j$ is a linearly independent collection of monomial expressions involving elements of $E$ and partial compositions, and $\lambda_j^i\in\bfk-\{0\}$.
\begin{defn} \label{defi:Prond}
The \textbf{set operadic equivalence relation generated by $\mathcal R$} is the finest equivalence relation \cR on $\mathbf M_E$, compatible with the set operad structure, such that
$$\mu_p^i\hbox{\cR}\mu_q^i\hbox{ for any }i\in\{1,\ldots,N\}\hbox{ and }p,q\in\{1,\ldots,k_i\}.$$
The \textbf{set operad associated to $\mathcal P$} is the set operad
$$\hbox{\cP}:=\mathbf M_E\Big/\hbox{\cR}.$$
\end{defn}
\begin{remark}
The set operad \cP depends on the presentation chosen for the linear operad $\mathcal P$. When $\mathcal P$ is given by the linearization of a set operad $\mathbf P$, we have $\hbox{\cP}=\mathbf P$. 
\end{remark}
\begin{remark}
Let $\mathbf{Q}$ be a quadratic set operad, and let $\mathcal{P}$ be the Koszul dual \cite{GK} of its linearization.
If $E=\mathbf{Q}_2$, the free set-operad $\mathbf{M}_E$ generated by $E$ is combinatorially represented
by binary trees with labeled leaves, and vertices decorated by elements of $E$. Let $\sim$ be
the equivalence on $\mathbf{M}_E(3)$ such that $\mathbf{M}_E(3)/\!\!\sim\,\,=\mathbf{Q}(3)$. 
By definition of the Koszul dual, $\mathcal{P}$ is generated by $E$, and the relations
\[\sum_{T\in C} \pm T=0,\]
where $C$ is a class of $\sim$ and the signs $\pm$ depend only of the form of the tree.
Applying  Definition \ref{defi:Prond}, we obtain that $\hbox{\cR}=\sim$, so in this case, $\hbox{\cP}=\mathbf{Q}$.
This holds for example if $\mathbf{Q}$ is the associative, or permutative, or diassociative operad: then $\mathcal{P}$ is the operad of, respectively,  associative,  or  pre-Lie, or dendriform algebras, 
with their usual presentations.
\end{remark}

\begin{prop}
Let $\Omega$ be a set endowed with a \cP-algebra structure. Considering the $\mathbf M_E$-algebra structure on $\Omega$ given by the operad morphism $\hbox{\cpi}:\mathbf M_E\surj{5}\hbox{\cP}$,
\begin{enumerate}
\item The colored subspecies $\widetilde{\mathbf M}_E^\Omega$ of $\mathbf M_E^\Omega$ defined by
$$(\widetilde{\mathbf M}_E^\Omega)_{A,\underline\alpha,\omega}:=\{\mu\in (M_E)_A,\, \mu(\underline\alpha)=\omega\}$$
is a set colored suboperad of $\mathbf M_E^\Omega$.
\item The colored subspecies $\widetilde{\mathcal M}_E^\Omega$ of $\mathcal M_E^\Omega$ defined by
$$(\widetilde{\mathcal M}_E^\Omega)_{A,\underline\alpha,\omega}:=\bfk.(\widetilde{\mathbf M}_E^\Omega)_{A,\underline\alpha,\omega}$$
is a linear colored suboperad of $\mathcal M_E^\Omega$.
\item The colored subspecies $\mathcal I$ of $\mathcal M_E^\Omega$ defined by
$$\mathcal I_{A,\underline\alpha,\omega}:=\bfk.\{\mu\in (\mathbf M_E)_{A},\, \mu(\underline\alpha)\neq \omega\}$$
is a right colored operadic ideal, and the quotient $\mathcal M_E/\mathcal I$ is isomorphic to $(\widetilde{\mathcal M}_E^\Omega)$ as a colored species.
\end{enumerate}
\end{prop}
\begin{proof}
Let $(A,\underline\alpha,\omega)$ and $(B,\underline\beta,\zeta)$ be two $\Omega$-colored finite sets. Let $\mu\in(\widetilde{\mathbf M}_E^\Omega)_{A,\underline\alpha,\omega}$ and $\nu\in(\widetilde{\mathbf M}_E^\Omega)_{B,\underline\beta,\zeta}$. We have then by definition of $\widetilde{\mathbf M}_E^\Omega$,
$$\mu(\underline\alpha)=\omega\hskip 3mm\hbox{ and }\hskip 3mm \nu(\underline\beta)=\zeta.$$
Now let $a\in A$. The partial composition $\mu\circ_a\nu$ is defined in the colored operad $\mathbf M_E^\Omega$ if and only if $\zeta=\underline\alpha(a)$. In that case we obviously have
$$\omega=\mu(\underline\alpha)=(\mu\circ_a\nu)(\underline\alpha\sqcup_a\underline\beta),$$
hence $\mu\circ_a\nu\in(\widetilde{\mathbf M}_E^\Omega)_{A\sqcup_aB,\, \underline\alpha\sqcup_a\underline\beta,\,\omega}$. The second assertion is an immediate consequence of the first. Now let $\mu\in(\mathbf M_E)_{A}$ and $\nu\in(\mathbf M_E)_{B}$ where $A$ and $B$ are two finite sets, and choose $a\in A$. The partial composition $\mu\circ_a\nu$ vanishes in $\mathcal M_E^\Omega$ unless the color matching condition $\zeta=\underline\alpha(a)$ is verified. If $\mu\in\mathcal I$, then by definition $\mu(\underline\alpha)\neq\omega$, hence 
$$
(\mu\circ_a\nu)(\underline\alpha\sqcup_a\underline\beta)=
\left\{
\begin{matrix}
\mu(\underline\alpha) &\hbox{ if } \zeta=\underline\alpha(a),\\
0 &\hbox{ if not},
\end{matrix}
\right.
$$
hence $\mu\circ_a\nu\in\mathcal I$. The last assertion is obvious from the definition.
\end{proof}
We denote by $\mathcal J$ the two-sided colored operadic ideal of $\mathcal M_E$ generated by $\mathcal I$. The following corollary is immediate:

\begin{coro}
Let $\pi$ be the projection from the free linear operad $\mathcal M_E$ onto $\mathcal P$, and let \cpi be the projection from $\mathcal M_E$ onto $\bfk$.\cP. Suppose that $\Omega$ is a \cP-algebra. Then the colored subspecies $\widetilde{\mathcal P}^\Omega:=\pi(\widetilde {\mathcal M}_E)$ of $\mathcal P^\Omega$ 
is a linear colored suboperad of $\mathcal P^\Omega$, and $\pi(\mathcal J)$ is a two-sided colored operadic ideal of $\mathcal P^\Omega$. 
\end{coro}
\begin{defn}
The colored operad $\wwt{\mathcal P}^\Omega:=\mathcal P/\pi(\mathcal J)$ is the \textbf{color-mixing operad} associated to the operad $\mathcal P$. It does depend on the presentation $\mathcal P=\mathcal M_E/\mathcal R$, and supposes a \cP-algebra structure on $\Omega$.
\end{defn}
\begin{remark}
The notion of color-mixing operad was already approached in the case when $\Omega$ is a commutative semigroup~: an $\Omega$-colored operad in which the output color is the sum of the input colors was given the name \textsl{current-preserving operad} in \cite{S2014}. The colored suboperad $\widetilde{\mathcal P}^\Omega$ associated to any ordinary operad $\mathcal P$ is an example.
\end{remark}
\begin{remark}
The right ideal $\mathcal I$ is not two-sided in general, hence the color-mixing operad $\wwt{\mathcal P}^\Omega$ is in general a proper quotient of the colored suboperad $\widetilde{\mathcal P}^\Omega$.
\end{remark}
\subsection{Graded algebras over a color-mixing operad and family structures}
Let $\Omega$ be a set, let $\bfk$ be a field, and let $\mathcal P$ be an operad in the category of $\bfk$-vector spaces. We keep the notations of the previous paragraphs, and in particular we fix a finite presentation $\mathcal P=\mathcal M_E/\mathcal R$.
\begin{defn}
An \textbf{$\Omega$-graded $\mathcal P$-algebra} is an algebra over the $\Omega$-colored operad $\wwt{\mathcal P}^\Omega$, i.e. an $\Omega$-graded $\bfk$-vector space $\mathcal V$ together with a morphism of colored operads $\Phi:\wwt{\mathcal P}^\Omega\to\hbox{End}(\mathcal V)$.
\end{defn}
\noindent Let us remark that the notion of $\Omega$-graded $\mathcal P$-algebra depends on the presentation of the operad $\mathcal P$.
\begin{prop}\label{omega-graded}
Any $\Omega$-graded $\mathcal P$-algebra $\mathcal V$  is an algebra over both colored operads $\mathcal P^\Omega$ and $\widetilde{\mathcal P}^\Omega$.
\end{prop}
\begin{proof}
It is an immediate consequence of the following diagram of $\Omega$-colored operads:
$$\widetilde{\mathcal P}^\Omega\inj{8}{\mathcal P}^\Omega\surj{8}\wwt{\mathcal P}^\Omega\longrightarrow \hbox{End}(\mathcal V).$$
\end{proof}
We are now ready to define $\Omega$-family $\mathcal P$-algebras, also called $\Omega$-relative $\mathcal P$-algebras in M. Aguiar's terminology \cite[Definition 14]{A2020}:
\begin{defn}
An \textbf{$\Omega$-family $\mathcal P$-algebra} is an $\Omega$-graded $\mathcal P$-algebra for which the underlying $\Omega$-graded object is uniform.
\end{defn}
\noindent Again, this notion depends on the presentation of $\mathcal P$.
\begin{prop}
Any $\Omega$-family $\mathcal P$-algebra is an $\Omega$-graded vector space $\mathcal U(V)$, where $V$ is an algebra over the operad $\overline{\mathcal F}(\wwt{\mathcal P}^\Omega)$.
\end{prop}
\begin{proof}
By definition, an $\Omega$-family $\mathcal P$-algebra is given by a vector space $V$ and a colored operad morphism
$$\Phi:\wwt{\mathcal P}^\Omega\longrightarrow \hbox{End}\big(\mathcal U(V)\big).$$
We have $ \hbox{End}\big(\mathcal U(V)\big)=\mathcal U(\hbox{End} V)$ by Equation \eqref{end-graded}.  The functor $\mathcal U$ of the left- (resp. right-) hand side is defined in Paragraph \ref{par:grad} (resp. \ref{par:colored}). The functor $\mathcal U$ is right-adjoint to the completed forgetful functor $\overline{\mathcal F}$ defined in Paragraph \ref{par:colored}, hence there is a morphism of ordinary operads from $\overline{\mathcal F}(\wwt{\mathcal P}^\Omega)$ to $\hbox{End}V$.
\end{proof}
Finally, we recover the close link between algebras and family algebras which was already observed on the known examples, and established by M. Aguiar in the case when $\Omega$ is a semigroup \cite[Paragraph 2.4]{A2020}:
\begin{prop}
Let $\mathcal V=\mathcal U(V)$ be an $\Omega$-family $\mathcal P$-algebra. Then the vector space $\mathcal F\mathcal U(V)=V\otimes\bfk\Omega$ is a $\mathcal P$-algebra.
\end{prop}
\begin{proof}
From Proposition \ref{omega-graded}, $\mathcal V=\mathcal U(V)$ is an algebra over $\mathcal P^\Omega=\mathcal U(\mathcal P)$. We have then a colored operad morphism
$$\Phi:\mathcal U(\mathcal P)\longrightarrow \hbox{End}\big(\mathcal U(V)\big)=\mathcal U(\hbox{End} V).$$
Now the functor $\mathcal U$ is left-adjoint to the forgetful functor $\mathcal F$, hence there is an operad morphism
$$\Psi:\mathcal P\longrightarrow \mathcal F\mathcal U(\hbox{End} V).$$
We conclude by the following observation: for any finite set $A$ we have
\begin{eqnarray*}
\big(\mathcal F\mathcal U(\hbox{End} V)\big)_A&=&\bigoplus_{\underline\alpha\in \Omega^A,\,\omega\in\Omega}(\hbox{End} V)_A\\
&\subset& \bigoplus_{\omega\in\Omega}\prod_{\underline\alpha\in \Omega^A}(\hbox{End} V)_A\\
&=&\hbox{End}\big(\mathcal F\mathcal U(V)\big)_A.
\end{eqnarray*}
These inclusions $j_A$ yield an operad morphism $j$, hence $j\circ\Psi$ is an operad morphism from $\mathcal P$ to $\hbox{End}\big(\mathcal F\mathcal U(V)\big)$.
\end{proof}
\noindent {\bf Acknowledgments}:
The third author is supported by the National Natural Science Foundation
of China (Grant No.\@ 11771191 and 11861051).

\end{document}